\RequirePackage{fix-cm}
\documentclass[smallextended]{svjour3}       % onecolumn (second format)
\smartqed  % flush right qed marks, e.g. at end of proof
\pdfoutput=1
\usepackage{graphicx}
 \usepackage{mathptmx}      % use Times fonts if available on your TeX system
\usepackage{geometry}
\usepackage{amsfonts}
\usepackage[english]{babel}
\usepackage[latin1]{inputenc}
\usepackage{rotating}
\usepackage{tikz}
\usepackage{graphics}
\usepackage{times}
\usepackage[T1]{fontenc}
\usepackage{amssymb}
\usepackage{amsmath}
\usepackage{mathrsfs}
\usepackage{latexsym}
\usepackage{euscript}
\usepackage{eufrak}
\usepackage{amscd}
\usepackage{dsfont}
\usepackage{enumitem}
\usepackage{cancel}
\usepackage[nice]{nicefrac}
\usepackage[toc,page]{appendix}
\usepackage[hidelinks,draft=false]{hyperref}
\usepackage{changepage}
\usepackage[obeyDraft,textsize=footnotesize,colorinlistoftodos,bordercolor=orange]{todonotes}
\usepackage{tensor}
\usepackage[draft,authormarkup=none]{changes}

\hypersetup{pdfauthor={Flandoli Franco, Russo Francesco, Zanco Giovanni},
pdftitle={Infinite-dimensional calculus under weak spatial regularity
 of the processes}}

\providecommand{\U}[1]{\protect\rule{.1in}{.1in}}

\newcommand{\call}[1]{$\left( \ref{#1} \right)$}

\newcounter{todocounter}

\DeclareMathOperator{\tr}{Tr}
\DeclareMathOperator{\supp}{supp}

\newcommand{\bE}{\mathds{E}}
\newcommand{\bF}{\mathds{F}}

\newcommand{\bN}{\mathds{N}\,}
\newcommand{\bP}{\mathds{P}}

\newcommand{\bR}{\mathds{R}}

\newcommand{\rH}{\mathscr{H}}
\newcommand{\rJ}{\mathscr{J}}

\newcommand{\sF}{\mathcal{F}}

\newcommand{\sH}{\mathcal{H}}

\newcommand{\sJ}{\mathcal{J}}

\newcommand{\sT}{\mathcal{T}}

\newcommand{\ud}{\,\mathrm{d}}

\newcommand{\ind}{\mathds{1}}

\newcommand{\VV}{\Vert}

\newcommand{\xphi}[2]{\left(\begin{smallmatrix}#1 \\ #2\end{smallmatrix}\right)}

\renewcommand{\epsilon}{\varepsilon}
\renewcommand{\phi}{\varphi}

\newcommand{\etsa}{e^{(t-s)A}}

\newcommand{\Fen}{F_{\epsilon,n}}

\begin{document}

\title{Infinite-dimensional calculus under weak spatial regularity
 of the processes%\thanks{Grants or other notes
%about the article that should go on the front page should be
%placed here. General acknowledgments should be placed at the end of the article.}
}
%\subtitle{Do you have a subtitle?\\ If so, write it here}

%\titlerunning{Short form of title}        % if too long for running head

\author{Franco Flandoli \and Francesco Russo \and Giovanni Zanco}

%\authorrunning{Short form of author list} % if too long for running head

\institute{F. Flandoli \at
              Universit\`a di Pisa, Dipartimento di Matematica, Largo Bruno Pontecorvo 5, 56127, Pisa (Italia) \\
%             \emph{Present address:} of F. Author  %  if needed
           \and
           F. Russo \at
              ENSTA ParisTech, Universit\'e Paris-Saclay, Unit\'e de Math\'ematiques appliqu\'ees, 828, boulevard des Mar\'echaux, F-91120 Palaiseau (France)\\
           \and
           G. Zanco \at
              Universit\`a di Pisa, Dipartimento di Matematica, Largo Bruno Pontecorvo 5, 56127, Pisa (Italia) \\
              Tel: +39 050 2213237
              \email{giovanni.zanco@gmail.com}              
}

\date{Received: date / Accepted: date}
% The correct dates will be entered by the editor

\maketitle

\begin{abstract}
  Two generalizations of It\^o formula to infinite-dimensional spaces are given. The first one, in Hilbert spaces, extends the classical one by taking advantage of cancellations, when they occur in examples and it is applied to the case of a group generator. The second one, based on the previous one and a limit procedure, is an It\^o formula in a special class of Banach spaces, having a product structure with the noise in a Hilbert component; again the key point is the extension due to a cancellation. This extension to Banach spaces and in particular the specific cancellation are motivated by path-dependent It\^o calculus.
% \PACS{PACS code1 \and PACS code2 \and more}
\end{abstract}

\section{Introduction}
\label{sec:intro}
Stochastic calculus and in particular It\^{o} formula has been extended since
long time ago from the finite to the infinite-dimensional case. The final
result in infinite dimensions is rather similar to the finite-dimensional one
except for some important details. One of them is the fact that unbounded
(often linear) operators usually appears in infinite dimensions and It\^{o}
formula has to cope with them. If the It\^{o} process $X\left(  t\right)  $,
taking values in a Hilbert space $H$, satisfies an identity of the form
\begin{equation}
\label{1}dX\left(  t\right)  =AX\left(  t\right)  dt+B\left(  t\right)  dt+C\left(
t\right)  dW\left(  t\right), 
%dX\left(  t\right)  =AX\left(  t\right)  dt+... %
\end{equation}
where $A:D\left(  A\right)  \subset H\rightarrow H$ is an unbounded linear
operator, then in the It\^{o} formula for a functional $F:\left[  0,T\right]
\times H\rightarrow\mathds{R}$ we have the term%
\begin{equation}
\left\langle AX\left(  t\right)  ,D_{x}F\left(  t,X\left(  t\right)  \right)
\right\rangle, \label{2}%
\end{equation}
which requires $X\left(  t\right)  \in D\left(  A\right)  $ to be defined. The
fact that $AX\left(  t\right)  $ also appears in (\ref{1}) is not equally
dramatic:\ equation (\ref{1}) could be interpreted in weak or mild form,
depending on the case. But the term (\ref{2}) is less flexible. Sometimes it
helps to interpret it as $\left\langle X\left(  t\right)  ,A^{\ast}%
D_{x}F\left(  t,X\left(  t\right)  \right)  \right\rangle $ or similar
reformulations, but this requires strong assumptions on $F$. Thus, in general,
direct application of It\^{o} formula is restricted to the case when $X\left(
t\right)  \in D\left(  A\right)  $. Among the ways to escape this difficulty
let us mention the general trick to replace (\ref{1}) by a regularized
relation of the form $dX_{n}\left(  t\right)  =AX_{n}\left(  t\right)  dt+...$
where $X_{n}\left(  t\right)  \in D\left(  A\right)  $ (the proof of Theorem
2.1 below is an example) and the so called mild It\^{o} formula proved in \cite{DPJR}. Another example of mild It\^{o} formula, under the assumptions
that $DF \in D\left(  A^*\right)$ was the object of \cite{FabbriRusso},
 Theorem 4.8.

Less common but important for some classes of infinite-dimensional
problems, for example the so-called path-dependent problems, is the phenomenon that
$\frac{\partial F}{\partial t}\left(  t,x\right)  $ exists only when $x$ lives
in a smaller space than $H$, for instance $D\left(  A\right)  $ (we shall
clarify this issue in the examples of section \ref{sec:infinite_reformulation} below). And in It\^{o}
formula we have the term
\[
\frac{\partial F}{\partial t}\left(  t,X\left(  t\right)  \right),
\]
so again we need the condition $X\left(  t\right)  \in D\left(  A\right)  $.

The purpose of this paper is to give a generalization of It\^{o} formula in
infinite-dimensional Hilbert and Banach spaces which solves the difficulties described above when the two
terms%
\[
\left\langle AX\left(  t\right)  ,D_{x}F\left(  t,X\left(  t\right)  \right)
\right\rangle \text{\qquad and\qquad}\frac{\partial F}{\partial t}\left(
t,X\left(  t\right)  \right)
\]
compensate each other when they sum, although they are not well defined
separately. This happens in a number of cases related to hyperbolic equations
and path-dependent problems. This gives rise to a new It\^{o} formula in which the term
\[
\frac{\partial F}{\partial t}\left(  s,X(s)\right)  +\left\langle AX(s),D_{x}%
F\left(  s,X(s)\right)  \right\rangle,
\]
that is a priori defined only when $X(s)\in D(A)$ is replaced by a term
\[
G(s,X(s)),
\]
where $G\left(  t,x\right)  $ is an extension of $\frac{\partial F}{\partial t}\left(  t,x\right)  +\left\langle Ax,D_{x}F\left(  t,x\right)  \right\rangle$.

In this introduction, for simplicity of exposition, we have insisted on the formula in a Hilbert space H, but one of the main purposes of our work is the further extension to a suitable class of Banach spaces, motivated by precise applications. Since the notations in the Banach case require more preliminaries, we address to section \ref{sec:ItoBanach} for this generalization.

It\^o formulae for this kind of problems, both at the abstract level and in applications to path-dependent functionals, have been investigated by many authors, see \cite{DGR2012,DGROsaka2014}, \cite{d}, \cite{CF3,CF1}, \cite{Neerven1}; however the idea to exploit the compensation explained above appears to be new and could be relevant for several applications.

Related to these problems is also the study of Kolmogorov equations in Banach spaces, see for instance \cite{DPC}, \cite{M07}, \cite{FZ15}, \cite{ZPhd}.

The paper is organized as follows. In section 2 we give a first generalization
of It\^{o} formula in Hilbert spaces. It serves as a first step to prove a
generalization to Banach spaces, described in section 3. But this also applies
directly to examples of hyperbolic SPDEs, as described in section 4. Finally,
in sections 5 and 6, we apply the extension to Banach spaces to several
path-dependent problems: in section 5 we consider typical path-dependent
functionals; in section 6 we deal with  the important case when $F\left(
t,x\right)  $ satisfies an infinite-dimensional Kolmogorov equation.

\section{An It\^{o} formula in Hilbert spaces}
\label{sec:ItoHilbert}

Let $H,U$ be two separable Hilbert spaces (which will be always identified
with their dual spaces) and $A:D\left(  A\right)  \subset
H\rightarrow H$ be the infinitesimal generator of a strongly continuous
semigroup $e^{tA}$, $t\geq0$, in $H$. 
Let $\left(\Omega,\mathcal{F},\bP\right)  $ be a complete probability space, $\mathbb{F}%
=\left(  \mathcal{F}_{t}\right)  _{t\geq0}$ be a complete filtration and
$\left(  W\left(  t\right)  \right)  _{t\geq0}$ be a Wiener process in $U$
with nuclear covariance operator $Q$; we address to \cite{DPZrosso} for a
detailed description of these notions of stochastic calculus in Hilbert
spaces.\\
Let $B:\Omega\times\left[  0,T\right]  \rightarrow H$ be a
progressively measurable process with $\int_{0}^{T}\left\vert B\left(
s\right)  \right\vert ds<\infty$ a.s., $C:\Omega\times\left[  0,T\right]
\rightarrow L\left(  U,H\right)  $ be progressively measurable with
$\int_{0}^{T}\left\Vert C\left(  s\right)  \right\Vert _{L\left(  U,H\right)
}^{2}ds<\infty$ a.s. and $X^{0}:\Omega\rightarrow H$ be a random vector,
measurable w.r.t. $\mathcal{F}_{0}$; here $L\left(  U,H\right)  $ denotes the
space of bounded linear operators from $U$ to $H$, with the corresponding norm
$\left\Vert \cdot\right\Vert _{L\left(  U,H\right)  }$.

Let $X=\left(
X\left(  t\right)  \right)  _{t\in\left[  0,T\right]  }$ be the stochastic
process in $H$ defined by%
\begin{equation}
\label{eq:XH}
X\left(  t\right)  =e^{tA}X^{0}+\int_{0}^{t}e^{\left(  t-s\right)  A}B\left(
s\right)  ds+\int_{0}^{t}e^{\left(  t-s\right)  A}C\left(  s\right)  dW\left(
s\right),
\end{equation}
formally solution to the equation%
\begin{equation}
\label{eq:SDEstrong}
dX\left(  t\right)  =AX\left(  t\right)  dt+B\left(  t\right)  dt+C\left(
t\right)  dW\left(  t\right)\ ,\quad X(0)=X^0\  .
\end{equation}

We are interested here in examples where $X\left(  s\right)  \notin D\left(
A\right)  $ and also, for a given function $F:[0,T]\times H\to\bR$, the derivative $\frac{\partial F}{\partial s}\left(  s,x\right)  $
exists only for a.e. $s$ and for $x\in D\left(  A\right)  $. In these cases the two terms
$\frac{\partial F}{\partial s}\left(  s,X\left(  s\right)  \right)  $ and
$\left\langle AX\left(  s\right)  ,DF\left(  s,X\left(  s\right)  \right)
\right\rangle $ have no meaning, in general.\\
However, there are examples where the sum $\frac{\partial F}{\partial
s}\left(  s,X\left(  s\right)  \right)  +\left\langle AX\left(  s\right)
,DF\left(  s,X\left(  s\right)  \right)  \right\rangle $ has a meaning even if
the two addends do not, separately. This paper is devoted to this class of
examples.\\
We assume that there exists a Banach space $\widetilde E$ continuously embedded in $H$ such that
\begin{enumerate}[label=\emph{(\Roman*)}]
\item $D(A)\subset \widetilde E$;
\item $e^{tA}$ is strongly continuous in $\widetilde E$;
\item $X(t)\in\widetilde E$;
\item almost surely the set $\left\{X(t)\right\}_{t\in[0,T]}$ is relatively compact in $\widetilde E$.
\end{enumerate}
The space $\widetilde E$ can possibly coincide with the whole space $H$ but in general it is a smaller space endowed with a finer topology and it is not required to be a inner product space.\\
In the setting described above, our abstract result is the following one.
 \begin{theorem}
 \label{thm:Ito_Hilbert}Let $F\in C\left(\left[  0,T\right]  \times H;\mathds{R}\right)$ be twice differentiable with respect to its second variable, with $DF\in C\left(  \left[  0,T\right]  \times H;H\right)  $ and $D^2F\in C\left(  \left[  0,T\right]  \times H;L\left(H,H\right)  \right)  $.
and assume the time derivative $\frac{\partial F}{\partial t}(t,x)$ exists for $(t,x)\in\sT\times D(A)$ where $\sT\subset [0,T]$ has Lebesgue measure $\lambda\left(\sT\right)=T$ and does not depend on $x$. 
 Assume moreover that there exists a continuous function $G:\left[  0,T\right]  \times \widetilde E\rightarrow\mathds{R}$ such that
 \begin{equation*}
 G\left(  s,x\right)  =\frac{\partial F}{\partial s}\left(  s,x\right)
 +\left\langle Ax,DF\left(  s,x\right)  \right\rangle \qquad\text{for all }(t,x)\in\sT\times D\left(  A\right). 
\end{equation*}
Let $X$ be the process defined in \eqref{eq:XH}.
Then%
 \begin{align*}
 F\left(  t,X\left(  t\right)  \right)   &  =F\left(  0,X^{0}\right)  +\int%
 _{0}^{t}G\left(  s,X\left(  s\right)  \right)  ds\\
 &  +\int_{0}^{t}\left(  \left\langle B\left(  s\right)  ,DF\left(  s,X\left(
 s\right)  \right)  \right\rangle +\frac{1}{2}Tr\left(  C\left(  s\right)
 QC\left(  s\right)  ^{\ast}D^{2}F\left(  s,X\left(  s\right)  \right)
 \right)  \right)  ds\\
 &  +\int_{0}^{t}\left\langle DF\left(  s,X\left(  s\right)  \right)  ,C\left(
 s\right)  dW\left(  s\right)  \right\rangle ,
 \end{align*}
where $DF$ and $D^2F$ denote the first and second Fréchet differentials of $F$ with respect to its second variable (the same notation will be used everywhere in this article).
 \end{theorem}
For the proof we need a preliminary result, namely a ``traditional'' It\^o formula that holds when $F$ is smooth.
\begin{proposition}
\label{prop:ItoFrancesco}
Let $\beta:\Omega\times [0,T]\to H$ and $\theta:\Omega\times[0,T]\to L(U,H)$ be two progressively measurable processes such that $\left\vert\beta(s)\right\vert$ and $\left\Vert\theta(s)\right\Vert_{L(U,H)}^2$ are integrable on $[0,T]$ a.s.; consider the It\^o process $Z$ in $H$ given by
\begin{equation*}
  Z(t)=Z^0+\int_0^t\beta(s)\ud s+\int_0^t\theta(s)\ud W(s)\ .
\end{equation*}
If $F\in C^{1,2}\left(  \left[  0,T\right]  \times H\right)$
% and $X\left(  t\right)  \in D\left(  A\right)  $ for a.e. $t\in\left[  0,T\right]  $, with $\left\vert AX\left(  \cdot\right)\right\vert $ integrable $\bP$-a.s., 
 the following identity holds (in probability):
\begin{align*}
F\left(  t,X\left(  t\right)  \right)   &  =F\left(  0,X^{0}\right)  +\int%
_{0}^{t}\frac{\partial F}{\partial s}\left(  s,X\left(  s\right)
\right)\ud s\\
&  +\int_{0}^{t}\left(  \left\langle \beta\left(  s\right)  ,DF\left(  s,X\left(
s\right)  \right)  \right\rangle +\frac{1}{2}Tr\left(  \theta\left(  s\right)
Q\theta\left(  s\right)  ^{\ast}D^{2}F\left(  s,X\left(  s\right)  \right)
\right)  \right)  ds\\
&  +\int_{0}^{t}\left\langle DF\left(  s,X\left(  s\right)  \right)  ,\theta\left(
s\right)  dW\left(  s\right)  \right\rangle\ .
\end{align*}
\end{proposition}
\begin{proof}
According to \cite{DGROsaka2014} we have that
  \begin{align}
    \label{eq:Francesco1}
    F\left(t,X(t)\right)&=F\left(0,X(0)\right)+\int_0^t\left\langle DF\left(s,X(s)\right),\ud^-X(s)\right\rangle  \nonumber\\
    &+\int_0^t\frac{\partial F}{\partial t}\left(s,X(s)\right)\ud s+\frac{1}{2}\int_0^tD^2F\left(s,X(s)\right)\ud\widetilde{\left[X,X\right]}(s),
  \end{align}
  where $\ud^-X$ denotes the integral via regularization introduced in \cite{DGROsaka2014}. We remark that $\widetilde{\left[X,X\right]}$ is here the global quadratic variation of the process in \call{eq:XH}.\\
By theorem 3.6 and proposition 3.8 of \cite{FabbriRusso} we get
\begin{multline*}
  \int_0^t\left\langle DF\left(s,X(s)\right),\ud^-X(s)\right\rangle\\
=\int_0^t\left\langle DF\left(s,X(s)\right),C(s)\ud W(s)\right\rangle+\int_0^t\left\langle DF\left(s,X(s)\right),AX(s)+B(s)\right\rangle\ud s\ .
\end{multline*}
By section 3.3 in \cite{DGR}
\begin{equation*}
  \left[X,X\right]^{\ud z}(t)=\int_0^tC(s)Q^{\nicefrac{1}{2}}\left(C(s)Q^{\nicefrac{1}{2}}\right)^\ast\ud s,
\end{equation*}
where $\left[X,X\right]^{\ud z}$ is the Da Prato-Zabczyk quadratic variation; hence proposition 6.12 of \cite{DGR} implies that
\begin{equation*}
  \int_0^tD^2F\left(s,X(s)\right)\ud\widetilde{\left[X,X\right]}(s)=\int_0^t\tr\left[D^2F\left(s,X(s)\right)C(s)Q^{\nicefrac{1}{2}}\left(C(s)Q^{\nicefrac{1}{2}}\right)^\ast\right]\ud s\ .
\end{equation*}
This concludes the proof.
\end{proof}

\begin{proof}[Proof of theorem \ref{thm:Ito_Hilbert}]
Let $\left\{\rho_\epsilon\right\}_{\epsilon\in(0,1]}$, $\rho_\epsilon:\bR\to\bR$, be a family of mollifiers with $\supp(\rho_\epsilon)\subseteq [0,1]$ for every $\epsilon$. For $x\in H$ set $F(t,x)=F(0,x)$ if $t\in[-1,0)$ and $F(t,x)=F(T,x)$ if $t\in(T,T+1]$.\\
Denote by $J_n$ the Yosida approximations $J_{n}=n\left(  n-A\right)^{-1}:H\rightarrow D\left(  A\right)  $, defined for every $n\in\mathds{N}$, which satisfy $\lim_{n\rightarrow\infty}J_{n}x=x$ for every $x\in H$. One also has $\lim_{n\rightarrow\infty}J_{n}^{\ast}x=x$, $\lim_{n\rightarrow\infty}J_{n}^{2}x=x$ and $\lim_{n\rightarrow\infty}\left(  J_{n}^{2}\right)  ^{\ast}x=x$ for every $x\in H$, used several times below, along with the fact that the operators\ $J_{n}$ and $J_{n}^{\ast}$ are equibounded. All these facts are well known and can be found also in \cite{DPZrosso}. Moreover it is easy to show that the family $J_n^2$ converges uniformly on compact sets to the identity (in the strong operator topology). Since $A$ generates a strongly continuous semigroup in $\widetilde E$ as well, all the properties of $J_n$ and $J_n^2$ just listed hold also in $\widetilde E$ (with respect to its topology).\\
Define now $\Fen:[0,T]\times H\to\bR$ as
\begin{equation*}
  \Fen(t,x)=\left(\rho_\epsilon\ast F\left(\cdot,J_nx\right)\right)(t)\ .
\end{equation*}
It is not difficult to show that $\Fen\in C^{1,2}\left([0,T]\times H;\bR\right)$.
% and that $\Fen$, $\frac{\partial \Fen}{\partial t}$, $D\Fen$ and $D^{2}\Fen$ are uniformly continuous on bounded sets; to this purpose,
Notice also that
\begin{align*}
\frac{\partial \Fen}{\partial t}\left(  t,x\right)   &  =\left(\dot\rho_\epsilon\ast F \left(  \cdot,J_{n}x\right)\right)(t)\ ,  \\
\left\langle D\Fen\left(  t,x\right)  ,h\right\rangle  &  =\left(\rho_\epsilon\ast\left\langle DF\left(  \cdot,J_{n}x\right)  ,J_{n}h\right\rangle\right)(t) \\
D^{2}\Fen\left(  t,x\right)  \left(  h,k\right)   &  =\left(\rho_\epsilon\ast D^{2}F\left(t,J_{n}x\right)  \left(  J_{n}h,J_{n}k\right)\right)(t) \  .
\end{align*}
Moreover
\begin{equation*}
  \frac{\partial\Fen}{\partial t}(t,x)=\left(\rho_\epsilon\ast\frac{\partial F}{\partial t}\left(\cdot,J_nx\right)\right)(t)
\end{equation*}
on $\sT\times D(A)$. To see this take $(t,x)\in \sT\times D(A)$, consider the limit
\begin{align}
\nonumber  \lim_{a\to 0}\frac{1}{a}\left[\Fen(t+a,x)-\Fen(t,x)\right]&=\lim_{a\to 0}\frac{1}{a}\int_\bR\rho_\epsilon(r)\left[F\left(t+a-r,J_nx\right)-F\left(t-r,J_nx\right)\right]\ud r \\
\label{eq:DFen2}                                                   &=\lim_{a\to 0}\frac{1}{a}\int_{B_\epsilon(0)}\rho_\epsilon(r)\left[F\left(t+a-r,J_nx\right)-F\left(t-r,J_nx\right)\right]\ud r
\end{align}
and set $R_\epsilon^t:=\left\{r\in B_\epsilon(0)\colon t-r\in \sT_0\right\}$, where $\sT_0:=[-1,0)\cup\sT\cup(T,T+1]$.\\
Since $t-R_\epsilon^t=\left(t-B_\epsilon(0)\right)\cap\sT_0$, we have that $\lambda\left(R_\epsilon^t\right)=\lambda\left(B_\epsilon(0)\right)$, hence we can go on from \call{eq:DFen2} finding
\begin{align*}
  \lim_{a\to 0}\frac{1}{a}\big[\Fen(t+a,x)-\Fen&(t,x)\big]=\lim_{a\to 0}\frac{1}{a}\int_{R_{\epsilon}^t}\rho_\epsilon(r)\left[F\left(t+a-r,J_nx\right)-F\left(t-r,J_nx\right)\right]\ud r\\
                                                          &=\int_{R_{\epsilon}^t}\rho_\epsilon(r)\frac{\partial F}{\partial t}\left(t-r,J_nx\right)\ud r\\
                                                          &=\left(\rho_\epsilon\ast\frac{\partial F}{\partial t}\left(\cdot,J_nx\right)\right)(t).
\end{align*}
Now set $X_{n}\left(  t\right)  =J_{n}X\left(  t\right)  $, $X_{n}^{0}=J_{n}X^{0}$, $B_{n}\left(  t\right)  =J_{n}B\left(  t\right)  $, $C_{n}\left(  t\right)=J_{n}C\left(  t\right)  $. Since $J_{n}$ commutes with $e^{tA}$, we have
\begin{equation*}
X_{n}\left(  t\right)  =e^{tA}X_{n}^{0}+\int_{0}^{t}e^{\left(  t-s\right)A}B_{n}\left(  s\right)  ds+\int_{0}^{t}e^{\left(  t-s\right)  A}C_{n}\left(s\right)  dW\left(  s\right)  \ .
\end{equation*}
Moreover, $X_{n}\left(  t\right)$, $B_n(t)$, $C_n(t)$ belong to $ D\left(  A\right)  $ for a.e. $t\in\left[  0,T\right]  $, with $\left\vert AX_{n}\left(  \cdot\right)\right\vert $ integrable $\bP$-a.s.; hence
\begin{equation*}
  X_n(t)=X_n^0+\int_0^t\left[AX_n(s)+B_n(s)\right]\ud s+\int_0^tC_n(s)\ud W(s)
\end{equation*}
and by the It\^{o} formula in Hilbert spaces given in proposition \ref{prop:ItoFrancesco} above we have
\begin{align*}
\Fen\left(t,X_{n}\left(t\right)\right) & =\Fen\left(0,X_{n}^{0}\right)+\int_{0}^{t}\left(\left\langle AX_{n}\left(s\right),D\Fen\left(s,X_{n}\left(s\right)\right)\right\rangle+\frac{\partial\Fen}{\partial s}\left(s,X_{n}\left(s\right)\right)\right)\ud s\\
&  +\int_{0}^{t}\left\langle B_{n}\left(s\right)  ,D\Fen\left(s,X_{n}\left(s\right)\right)\right\rangle \ud s+\int_{0}^{t}\left\langle D\Fen\left(s,X_{n}\left(s\right)\right),C_{n}\left(s\right)\ud W\left(s\right)\right\rangle \\
&  +\frac{1}{2}\int_{0}^{t}Tr\left[C_{n}\left(s\right)QC_{n}\left(s\right)^{\ast}D^{2}\Fen\left(s,X_{n}\left(s\right)\right)\right]\ud s.
\end{align*}
Let us prove the convergence (as $n\to\infty$ and $\epsilon\to 0$) of each term to the corresponding one of the formula stated by the theorem. We fix $t$ and prove the a.s. (hence in probability) convergence of each term, except for the convergence in probability of the It\^{o} term; this yields the conclusion. 

Given $\left(\omega,t\right)  $, we have $\Fen\left(t,X_{n}\left(\omega,t\right)\right)=\rho_\epsilon \ast F\left(\cdot,J_{n}^{2}X\left(\omega,t\right)\right)(t)$ and thus
\begin{align*}
\left\vert\Fen\left(t,X_n(\omega,t)\right)-F\left(t,X(\omega,t)\right)\right\vert&=\left\vert\int_\bR\rho_\epsilon(r)F\left(t-r,J_n^2X(\omega,t)\right)\ud r-F\left(t,X(\omega,t)\right)\right\vert\\
&\leq\int_{B_\epsilon(0)}\rho_\epsilon(r)\left\vert F\left(t-r,J_n^2X\left(\omega,t\right)\right)-F\left(t,X(\omega,t)\right)\right\vert\ud r,
\end{align*}
which is arbitrarily small for $\epsilon$ small enough and $n$ big enough, because $J_{n}^{2}$ converges strongly to the identity and $F$ is continuous;
similarly 
\begin{equation*}
\lim_{\substack{\epsilon\to 0\\n\rightarrow\infty}}\Fen\left(0,X_{n}^{0}\left(\omega\right)\right)=F\left(0,X^{0}\left(\omega\right)\right)\ .
\end{equation*}
From now on we work in the set $\Omega_1$ where $X$ has relatively compact paths in $\widetilde E$ (hence in $H$).
Fix $\delta>0$. Since for $\omega\in\Omega_1$ the set $\left\{X(\omega,s)\right\}_{s\in[0,t]}$ is relatively compact, we have that $J_n^2X(s)$ converges uniformly with respect $s$ to $X(s)$, hence there exists $N\in\bN$ such that for any $n>N$ $\left\vert J_n^2X(s)-X(s)\right\vert<\nicefrac{\delta}{2}$ for all $s$; moreover the set $\left\{J_nX(s)\right\}_{n,s}$ is bounded.\\
The family $\left\{B_{\nicefrac{\delta}{2}}\left(X(s)\right)\right\}_{s\in[0,t]}$ is an open cover of $\left\{X(s)\right\}_{s\in[0,t]}$; by compactness it admits a finite subcover $\left\{B_{\nicefrac{\delta}{2}}\left(X(s_i)\right)\right\}_{i=1,\dots,M}$ for some finite set $\left\{s_1,\dots,s_M\right\}\subset[0,t]$, therefore for any $s$ there exists $i\in\{1,\dots,N\}$ such that $\left\vert X(s)-X\left(s_i\right)\right\vert<\nicefrac{\delta}{2}$ and
  \begin{equation*}
    \left\vert J_n^2X(s)-X\left(s_i\right)\right\vert\leq\left\vert J_n^2X(s)-X(s)\right\vert+\left\vert X(s)-X\left(s_i\right)\right\vert<\delta
  \end{equation*}
for $n>N$ where $N$ does not depend on $s$ since the convergence is uniform. This shows that the set $\left\{J_n^2X(s)\right\}_{n,s}$ is totally bounded both in $\widetilde E$ and in $H$.\\
Therefore we can study the convergence of the other terms as follows. First we consider the difference
\begin{align*}
\left\vert\int_0^t\right.&\left\langle B_n(s),D\Fen\left(s,X_n(s)\right)\right\rangle\ud s-\left.\int_0^t\left\langle B(s),DF\left(s,X(s)\right)\right\rangle\ud s\right\vert\\
&\leq\left\vert\int_0^t\left\langle J_n^2B(s),\left(\rho_\epsilon\ast DF\left(\cdot,J_n^2X(s)\right)\right)(s)\right\rangle\ud s-\int_0^t\left\langle J_n^2B(s),DF\left(s,X(s)\right)\right\rangle\ud s\right\vert\\
&\phantom{\leq}+\left\vert\int_0^t\left\langle J_n^2B(s)-B(s),DF\left(s,X(s)\right)\right\rangle\ud s\right\vert.
\end{align*}
The second term in this last sum is bounded by
\begin{equation*}
  \int_0^t\left\vert J_n^2B(s)-B(s)\right\vert\left\vert DF\left(s,X(s)\right)\right\vert\ud s
\end{equation*}
and $\left\{X(s)\right\}_s$ is compact, hence $\left\vert DF\left(s,X(s)\right)\right\vert$ is bounded uniformly in $s$ and, since the $J_n^2$ are equibounded and converge strongly to the identity and $B$ is integrable, Lebesgue's dominated convergence theorem applies. The first term in the previous sum instead is bounded by
\begin{equation}
\label{eq:ItoHB2}
  \int_0^t\left\vert J_n^2B(s)\right\vert\int_{B_\epsilon(0)}\rho_\epsilon(r)\left\vert DF\left(s-r,J_n^2X(s)\right)-DF\left(s,X(s)\right)\right\vert\ud r\ud s\ ;
\end{equation}
by the discussion above the set $[0,t]\times\left(\left\{J_n^2X(s)\right\}_{n,s}\cup\left\{X(s)\right\}_s\right)$ is contained in a compact subset of $[0,T]\times H$, hence $\left\vert DF\right\vert$ is bounded on that set uniformly in $s$ and $r$. Thanks again to the equicontinuity of the operators $J_n^2$ and the integrability of $B$, \call{eq:ItoHB2} is shown to go to $0$ by the dominated convergence theorem and the continuity of $DF$.\\
About the critical term involving $G$ we have
\begin{align*}
  \left\vert\int_0^t\right.&\left(\frac{\partial\Fen}{\partial t}\left(s,X_n(s)\right)+\left\langle AX_n(s),D\Fen\left(s,X_n(s)\right)\right\rangle\right)\ud s-\left.\int_0^t G\left(s,X(s)\right)\ud s\right\vert\\
  &\leq\int_{[0,t]\cap\sT}\left\vert\rho_\epsilon\ast\left(\frac{\partial F}{\partial t}\left(\cdot,J_n^2X(s)\right)+\left\langle AJ_n^2X(s),DF\left(\cdot,J_n^2X(s)\right)\right\rangle\right)(s)-G\left(s,X(s)\right)\right\vert\ud s\\
    &=\int_{[0,t]\cap\sT}\left\vert\left(\rho_\epsilon\ast G\left(\cdot,J_n^2X(s)\right)\right)(s)-G\left(s,X(s)\right)\right\vert\ud s\\
    &\leq\int_{[0,t]\cap\sT}\int_{B_\epsilon(0)}\rho_\epsilon(r)\left\vert G\left(s-r,J_n^2X(s)\right)-G\left(s,X(s)\right)\right\vert\ud r\ud s
\end{align*}
and this last quantity goes to $0$ by compactness and continuity of $G$ in the same way as the previous term (now with respect to the topology on $\widetilde E$).\\
For the It\^{o} term we have
\begin{multline}
  \label{eq:ItoCH}\int_0^t\left\vert C^\ast(s)\left(J_n^2\right)^\ast\left(\rho_\epsilon\ast DF\left(\cdot,J_n^2X(s)\right)\right)(s)-C^\ast(s)DF\left(s,X(s)\right)\right\vert^2\ud s\\
  \leq\int_0^t\left\Vert C(s)\right\Vert^2\left\vert\left( J_n^\ast\right)^2\rho_\epsilon\ast Df\left(\cdot,J_n^2X(s)\right)(s)-DF\left(s,X(s)\right)\right\vert^2\ud s\ ;
\end{multline}
writing
\begin{multline*}
  \left\vert\left( J_n^\ast\right)^2\rho_\epsilon\ast DF\left(\cdot,J_n^2X(s)\right)(s)-DF\left(s,X(s)\right)\right\vert\\
\leq \left\vert\left(J_n^\ast\right)^2\rho_\epsilon\ast DF\left(\cdot,J_n^2X(s)\right)(s)-\left(J_n^\ast\right)^2DF\left(s,X(s)\right)\right\vert+\left\vert\left(J_n^\ast\right)^2DF\left(s,X(s)\right)-DF\left(s,X(s)\right)\right\vert,
\end{multline*}
it is immediate to see that the right-hand side of \call{eq:ItoCH} converges to $0$ almost surely, hence
\begin{equation*}
  \int_0^t\left\langle D\Fen\left(s,X_n(s)\right),C_n(s)\ud W(s)\right\rangle\to\int_0^t\left\langle DF\left(s,X(s)\right),C(s)\ud W(s)\right\rangle
\end{equation*}
in probability.\\
It remains to treat the trace term. Let $\left\{h_j\right\}$ be an orthonormal complete system in $H$; then 
\begin{multline}
\label{eq:ItoTr1}
  \left\vert\int_0^t\tr\left[C_n(s)QC_n(s)^\ast D^2\Fen\left(s,X_n(s)\right)\right]\ud s-\int_0^t\tr\left[C(s)QC(s)^\ast DF\left(s,X(s)\right)\right]\ud s\right\vert\\
\leq\int_0^t\sum_j\left\vert\left\langle\left[J_nC(s)QC(s)^\ast\left(J_n^\ast\right)^2\rho_\epsilon\ast D^2F\left(\cdot,J_n^2X(s)\right)(s)J_n-C(s)QC(s)^\ast D^2F\left(s,X(s)\right)\right]h_j,h_j\right\rangle\right\vert\ud s.
\end{multline}
Now for any $j$
\begin{align*}
  \left\vert J_n\right. &C(s)QC(s)^\ast\left(J_n^\ast\right)^2\rho_\epsilon\ast D^2F\left(\cdot,J_n^2X(s)\right)(s)\left.J_nh_j-C(s)QC(s)^\ast D^2F\left(s,X(s)\right)h_j\right\vert\\
  &\leq\left\vert J_nC(s)QC(s)^\ast\left(J_n^\ast\right)^2\rho_\epsilon\ast D^2F\left(\cdot,J_n^2X(s)\right)(s)J_nh_j-C(s)QC(s)^\ast D^2F\left(s,X(s)\right)J_nh_j\right\vert\\
  &\phantom{\leq}+\left\Vert C(s)QC(s)^\ast D^2F\left(s,X(s)\right)\right\Vert\cdot\left\vert J_nh_j-h_j\right\vert\ .
\end{align*}
The second term in the sum converges to $0$ thanks to the properties of $J_n$; the first one is bounded by the sum
\begin{multline}
  \label{eq:ItoTr2}
  \left\vert J_nC(s)QC(s)^\ast\left(J_n^\ast\right)^2\rho_\epsilon\ast D^2F\left(\cdot,J_n^2X(s)\right)(s)J_nh_j-J_nC(s)QC(s)^\ast\left(J_n^\ast\right)^2D^2F\left(s,X(s)\right)J_nh_j\right\vert\\
+\left\vert\left[J_nC(s)QC(s)^\ast\left(J_n^\ast\right)^2-C(s)QC(s)^\ast\right]D^2F\left(s,X(s)\right)J_nh_j\right\vert,
\end{multline}
whose first addend is less or equal to
\begin{equation*}
  \left\Vert J_nC(s)QC(s)^\ast\left(J_n^\ast\right)^2\right\Vert\int_{B_\epsilon(0)}\rho_\epsilon\left\vert D^2F\left(s-r,J_n^2X(s)\right)-D^2F\left(s,X(s)\right)\right\vert\ \left\vert J_nh_j\right\vert\ud r,
\end{equation*}
which is shown to go to zero as before. For the second addend of \call{eq:ItoTr2} notice that for any $k\in H$
\begin{align*}
  \Big\vert\big[ J_nC(s)QC(s)^\ast&\left(J_n^\ast\right)^2-C(s)QC(s)^\ast\big]k\Big\vert\\
  &\leq\left\vert\left[J_nC(s)QC(s)^\ast\left(J_n^\ast\right)^2-J_nC(s)QC(s)^\ast\right]k\right\vert+\left\vert\left[J_nC(s)QC(s)^\ast-C(s)QC(s)^\ast\right]k\right\vert\\
  &\leq\left\Vert J_nC(s)QC(s)^\ast\right\Vert\ \left\vert\left(J_n^\ast\right)^2k-k\right\vert+\left\vert J_nC(s)QC(s)^\ast k-C(s)QC(s)^\ast k\right\vert,
\end{align*}
which tends to $0$ as $n$ tends to $\infty$.\\
The same compactness arguments used in the previous steps, the continuity of $D^2F$ and the equiboundedness of the family $\left\{J_n\right\}$ allow to apply Lebesgue's dominated convergence theorem both to the series and to the the integral with respect to $s$ in \call{eq:ItoTr1}.
This concludes the proof.
\end{proof}

\section{Extension to particular Banach spaces}
\label{sec:ItoBanach}
In this section we consider the following framework. Let $H_1$ be a separable Hilbert space with scalar product $\langle \cdot\rangle_1$ and norm $\VV \cdot\VV_1$ and let $E_2$ be a Banach space, with norm $\VV\cdot\VV_{E_2}$ and duality pairing denoted by $\langle\cdot,\cdot\rangle$, densely and continuously embedded in another separable Hilbert space $H_2$ with scalar product and norm denoted respectively by $\langle\cdot\rangle_2$ and $\VV\cdot\VV_2$. Then set $H:=H_1\times H_2$ so that
\begin{equation*}
  E:=H_1\times E_2\subset H
\end{equation*}
with continuous and dense embedding. We adopt here the standard identification of $H$ with $H^\ast$ so that
\begin{equation*}
  E\subset H\cong H^\ast\subset E^\ast\ \text{.}
\end{equation*}
%Finally let $\widetilde E\subseteq E$ be a closed subspace. 
Our aim here is to extend the results exposed so far to situations in which the process $X$ lives in a subset of $E$ but the noise only acts on $H_1$.\\
\begin{example}
  In the application of this abstract framework to path-dependent functionals (see section \ref{sec:infinite_reformulation}), we will choose the spaces
  \begin{equation*}
    H=\bR^d\times L^2\left(-T,0;\bR^d\right)
  \end{equation*}
and
\begin{equation*}
  E=\bR^d\times \left\{\phi\in C\left([-T,0);\bR^d\right)\colon \exists\lim_{s\to 0^-}\phi(s)\in\bR^d\right\}\ \text{.}
\end{equation*}
\end{example}
Similarly to the setup we introduced in section \ref{sec:ItoHilbert}, consider a complete probability space $\left(\Omega,\sF,\bP\right)$ with a complete filtration $\bF=\left(\sF_t\right)_{t\geq 0}$ and a Wiener process $\left( W(t) \right)_{t\geq 0}$ in another separable Hilbert space $U$ with nuclear covariance operator $Q$.\\
Consider a linear operator $A$ on $H$ with domain $D(A)\subset E$ and assume that it generates a strongly continuous semigroup $e^{tA}$ in $H$.
% defined both from $H$ to itself and from $E$ to itself (and consequently denoted on occasion also as $A_H$ and $A_E$), with possibly different domains $D\left(A_H\right)$ and $D\left(A_E\right)\subset D\left(A_H\right)$, 
%and assume that $$\overline{\Upsilon}^E=E$$. Notice that both $H_1\times \{0\}$ and $\widetilde E$ are contained in $\Upsilon$.
Let $B:\Omega\times [0,T]\to E$ be a progressively measurable process s.t. $\int_0^t\left\vert B(t)\right\vert\ud t<\infty$ as in section \ref{sec:ItoHilbert}; let then $\widetilde C:\Omega\times [0,T]\to L(U,H_1)$ be another progressively measurable process that satisfies $\int_0^T\left\VV C(t)\right\VV^2_{L(U,H_1)}\ud t<\infty$ and define $C:\Omega\times[0,T]\to L(U,E)$ as
\begin{equation*}
  C(t)u=\left(\begin{matrix}\widetilde C(t)u\\0\end{matrix}\right)\ \text{,}\ u\in U\ \text{;}
\end{equation*}
let $X^0$ be a $\sF_0$-measurable random vector with values in $H$ and set
\begin{equation}
\label{eq:SDE_Banach}
  X(t)=e^{tA}X^0+\int_0^te^{(t-s)A}B(s)\ud s+\int_0^te^{(t-s)A}C(s)\ud W(s)\ \text{.}
\end{equation}
%Notice that the stochastic integral is well defined as a Hilbert space stochastic integral and thus the process $X$ takes values in $H$ and has almost surely continuous trajectories in $H$; 
Finally set
 \begin{gather*}
   \widetilde E=\overline{D\left(A\right)}^E\ ,
   \widetilde D=A^{-1}(E)\ .
 %  \Upsilon=H_1\times \left\{x_2\colon \exists x_1\colon \xphi{x_1}{x_2}\in D\left(A_E\right)\right\}
 \end{gather*}

Notice that $\widetilde D\subset D(A)\subset \widetilde E$. In most examples the set $\widetilde D$ is not dense in $E$. 
As in section \ref{sec:ItoHilbert} we assume here that $e^{tA}$ is strongly continuous in $\widetilde E$ (and this in turn implies that $\widetilde D$ is dense in $\widetilde E$), $X(t)$ actually belongs to $\widetilde E$ and that almost surely the set $\left\{X(t)\right\}_{t\in[0,T]}$ is relatively compact in $E$.\\
\begin{example}
  In the path-dependent case, we will have
    \begin{equation*}
D\left(A\right)=\left\{\xphi{x_1}{x_2}\in H \colon x_2\in W^{1,2}\left(-T,0;\bR^d\right),\ x_1=\lim_{s\to 0^-}x_2(s)\right\}\ \text{,}
\end{equation*}
\begin{equation*}
A=\left(
\begin{array}
[c]{cc}%
0 & 0\\
0 & \frac{d}{dr}%
\end{array}
\right)\ \text{,}
\end{equation*}
 \begin{equation*}
 \widetilde E=\left\{\xphi{x_1}{x_2}\in E \colon x_1=\lim_{s\to 0^-}x_2(s)\right\}\ ,
 \end{equation*}
\begin{equation*}
\widetilde D=\left\{\xphi{x_1}{x_2}\in E \colon x_2\in C^1\left([-T,0);\bR^d\right),\ x_1=\lim_{s\to 0^-}x_2(s)\right\}\ \text{.}\\
\end{equation*}
\end{example}
Finally consider a sequence $\sJ_n$ of linear continuous operators, $\sJ_n:H\to E$ with the properties:
\begin{enumerate}[label=\emph{(\roman*)}]
%\item $\sJ_nx\in \Upsilon$ for every $x\in H$;
\item $\sJ_nx\in \widetilde D$ for every $x\in\widetilde E$;
\item $\sJ_nx\to x$ in the topology of $E$ for every $x\in E$;
\item $\sJ_n$ commutes with $A$ on $D(A)$.
\end{enumerate}
By Banach-Steinhaus and Ascoli-Arzelà theorems it follows that the operators $\sJ_n$ are equibounded and converge to the identity uniformly on compact sets of $E$.
\begin{theorem}
\label{thm:Ito_Banach}
  Assume there exists a sequence $\sJ_n$ as above and let $F\in C\left([0,T]\times E;\bR\right)$ be twice differentiable with respect to its second variable with $DF\in C\left([0,T]\times E;E^\ast\right)$ and $D^2F\in C\left([0,T]\times E;L\left(E;E^\ast\right)\right)$. Assume the time derivative $\frac{\partial F}{\partial t}(t,x)$ exists for $(t,x)\in\sT\times \widetilde D$ where $\sT\subset [0,T]$ has Lebesgue measure $\lambda\left(\sT\right)=T$ and does not depend on $x$. If there exists a continuous function $G:[0,T]\times \widetilde E\to \bR$ such that
  \begin{equation}
    \label{eq:G_banach}
    G(t,x)=\frac{\partial F}{\partial t}(t,x)+\langle Ax,DF(t,x)\rangle\quad\forall\ x\in \widetilde D\text{,}\;\forall\ t\in\sT\text{,}
  \end{equation}
then, in probability,
\begin{align*}
  F\left(t,X(t)\right)&=F\left(0,X^0\right)+\int_0^tG\left(s,X(s)\right)\ud s\\
  &+\int_0^t\left(\langle B(s),DF\left(s,X(s)\right)\rangle+\frac{1}{2}\tr_{H_1}\left[C(s)QC(s)^\ast D^2F\left(s,X(s)\right)\right]\right)\ud s\\
  &+\int_0^t\langle DF\left(s,X(s)\right),C(s)\ud W(s)\rangle\ \text{,}
\end{align*}
where $\tr_{H_1}$ is defined for $T\in L\left(E,E\right)$ as
\begin{equation*}
  \tr_{H_1}T=\sum_{j}\left\langle T\xphi{h_j}{0},\xphi{h_j}{0}\right\rangle\ \text{,}
\end{equation*}
$\left\{h_j\right\}$ being an orthonormal complete system in $H_1$.

\end{theorem}
\begin{proof}
  Set $F_n:[0,T]\times H\to \bR$, $F_n(t,x):=F\left(t,\sJ_nx\right)$. Thanks to the assumptions on $F$ we have that $F_n$ is twice differentiable with respect to the variable $x$ and 
  \begin{gather}
    \label{eq:DFn}DF_n(t,x)=\sJ_n^\ast DF\left(t,\sJ_nx\right)\in L\left(H;\bR\right)\cong H\\
    \label{eq:D2Fn}D^2F_n(t,x)=\sJ_n^\ast D^2F\left(t,\sJ_nx\right)\sJ_n\in L\left(H;H\right).
  \end{gather}
Furthermore for any $t\in\sT$ the derivative of $F_n$ with respect to $t$ is defined for all $x\in H$ and equals
\begin{equation}
    \label{eq:dtF}\frac{\partial F_n}{\partial t}(t,x)=\frac{\partial F}{\partial t}\left(t,\sJ_nx\right)\ \text{.}
  \end{equation}
Set $G_n:[0,T]\times \widetilde E\to\bR$, $G_n\left(t,x\right):=G\left(t,\sJ_nx\right)$. $G_n$ is obviously continuous; we check now that for any $t\in\sT$ $G_n(t,\cdot)$ extends $\frac{\partial F_n}{\partial t}(t,\cdot)+\langle A\cdot,DF_n(t,\cdot)\rangle$ from $D\left(A\right)$ to $\widetilde E$. Since $\sJ_n$ maps $\widetilde E$ into $\widetilde D\subset D\left(A\right)\subset H$ we have 
\begin{align*}
  G_n\left(t,x\right)&=G\left(t,\sJ_nx\right)\\
                    &=\frac{\partial F}{\partial t}\left(t,\sJ_nx\right)+\langle A\sJ_nx,DF\left(t,\sJ_nx\right)\rangle\ \text{;}\\
\intertext{if we choose $x\in D\left(A\right)$, $\sJ_n$ commutes with $A$ so that we can proceed to get}
                    &=\frac{\partial F}{\partial t}\left(t,\sJ_nx\right)+\langle \sJ_nAx,DF\left(t,\sJ_nx\right)\rangle\\
                    &=\frac{\partial F_n}{\partial t}(t,x)+\langle Ax,DF_n(t,x)\rangle\ \text{.}
\end{align*}
Notice that here only the term $\langle Ax,DF_n(t,x)\rangle$ has to be extended (since it is not well defined outside $D\left(A\right)$) while the time derivative of $F_n$ makes sense on the whole space $H$ by definition.\\
We can now apply theorem \ref{thm:Ito_Hilbert} to $F_n$ and $G_n$, obtaining that for each $n$
\begin{align*}
  F_n\left(t,X(t)\right)&=F_n\left(0,X^0\right)+\int_0^tG_n\left(s,X(s)\right)\ud s\\
  &+\int_0^t\left[\langle B(s),DF_n\left(s,X(s)\right)\rangle+\frac{1}{2}\tr\left[C(s)QC(s)^\ast D^2F_n\left(s,X(s)\right)\right]\right]\ud s\\
    &+\int_0^t\langle DF_n\left(s,X(s)\right),C(s)\ud W(s)\rangle\ \text{.}
\end{align*}
Here $C(s)QC(s)^\ast$ maps $E^\ast$ into $E$, therefore $C(s)QC(s)^\ast D^2F_n\left(s,X(s)\right)$ maps $H$ into $E\subset H$ and the trace term can be interpreted as in $H$. Also, since $C(s)$ belongs to $L\left(U;H_1\times\{0\}\right)$, we have that the stochastic integral above is well defined as a stochastic integral in a Hilbert space.\\
Substituting the definition of $F_n$ and identities \call{eq:DFn}, \call{eq:D2Fn} in the previous equation we get
\begin{align*}
  F&\left(t,\sJ_nX(t)\right)=F\left(0,\sJ_nX^0\right)+\int_0^t G\left(s,\sJ_nX(s)\right)\ud s\\
  &+\int_0^t\left[\langle \sJ_nB(s),DF\left(s,\sJ_nX(s)\right)\rangle+\frac{1}{2}\tr\left[C(s)QC(s)^\ast\sJ_n^\ast D^2F\left(s,\sJ_nX(s)\right)\sJ_n\right]\right]\ud s\\
  &+\int_0^t\langle DF\left(s,\sJ_nX(s)\right),\sJ_nC(s)\ud W(s)\rangle\ \text{.}
\end{align*}
Now we fix $(\omega,t)$ and study the convergence of each of the terms above. Since $X(\omega,t)\in \widetilde E$, $\sJ_nX(\omega,t)\to X(\omega,t)$ almost surely as $n\to\infty$ and therefore by continuity of $F$ we have that $F\left(t,\sJ_nX(\omega,t)\right)$ converges to $F\left(t,X(\omega,t)\right)$ almost surely. For the same reasons $F\left(0,\sJ_nX^0(\omega)\right)$ converges to $F\left(0,X^0(\omega)\right)$ almost surely.\\
Denote by $\Omega_1$ the set of full probability where each of the trajectories $\left\{X(\omega,t)\right\}_t$ is relatively compact. Arguing as in the proof of theorem \ref{thm:Ito_Hilbert} it can be shown that, thanks to the uniform convergence on compact sets of the $\sJ_n$, the set $\left\{\sJ_nX(\omega,t)\right\}_{n,t}$ is totally bounded in $E$ for any $\omega\in\Omega_1$. Therefore the a.s. convergence of the terms $\int_0^tG\left(s,\sJ_nX(\omega,s)\right)\ud s$ and $\int_0^t\langle\sJ_n B(\omega,s),DF\left(s,\sJ_nX(\omega,s)\right)\rangle\ud s$ follows from the dominated convergence theorem since $G$ and $DF$ are continuous, $B$ is integrable and the family $\left\{\sJ_n\right\}$ is equibounded.\\
To show the convergence of the stochastic integral term consider
\begin{multline}
\label{eq:stochInt}  \int_0^t\left\VV C(s)^\ast\sJ_n^\ast DF\left(s,\sJ_nX(s)\right)-C(s)^\ast DF\left(s,X(s)\right)\right\VV_U^2\ud s\\ \leq \int_0^t\left\VV C(s)\right\VV_{L(U,E)}^2\ \left\VV \sJ_n^\ast DF\left(s,\sJ_nX(s)\right)-DF\left(s,X(s)\right)\right\VV_{E^\ast}^2\ud s\ .
\end{multline}
%(since $E^\ast\subset \widetilde E^\ast$ with continuous embedding and $\left\VV C(s)\right\VV_{L\left(U,\widetilde E^\ast\right)}=\left\VV C(s)\right\VV_{L(U,E)}$ because $C(s)$ takes values in $H_1\times\{0\}$).
Now
\begin{align*}
  \big\VV&\sJ_n^\ast DF\left(s,\sJ_nX(s)\right)-DF\left(s,X(s)\right)\big\VV_{E^\ast}=\sup_{\substack{e\in E\\ \VV e\VV=1}}\left\vert\langle e,\sJ_n^\ast DF\left(s,\sJ_nX(s)\right)-DF\left(s,X(s)\right)\rangle\right\vert\\
  &=\sup_{\substack{e\in E\\ \VV e\VV =1}}\left\vert\langle \sJ_n e,DF\left(s,\sJ_n X(s)\right)\rangle-\langle e,DF\left(s,X(s)\right)\rangle\right\vert\\
  &\leq\sup_{\substack{e\in E\\ \VV e\VV =1}}\left[\left\vert\langle\sJ_n e,DF\left(s,\sJ_n X(s)\right)\rangle-\langle \sJ_n e,DF\left(s,X(s)\right)\rangle\right\vert+\left\vert\langle \sJ_n e-e,DF\left(s,X(s)\right)\rangle\right\vert\right]\\
  &\leq\sup_{\substack{e\in E\\ \VV e\VV =1}}\left[\left\VV\sJ_n\right\VV_E\ \left\VV DF\left(s,\sJ_nX(s)\right)-DF\left(s,X(s)\right)\right\VV_{E^\ast}+\left\VV\sJ_ne-e\right\VV_{E}\ \left\VV DF\left(s,X(s)\right)\right\VV_{E^\ast}\right]
\end{align*}
% To show the convergence of the stochastic integral term consider
% \begin{multline}
% \label{eq:stochInt}  \int_0^t\left\VV C(s)^\ast\sJ_n^\ast DF\left(s,\sJ_nX(s)\right)-C(s)^\ast DF\left(s,X(s)\right)\right\VV_U^2\ud s\\ \leq \int_0^t\left\VV C(s)\right\VV_{L(U,E)}^2\ \left\VV \sJ_n^\ast DF\left(s,\sJ_nX(s)\right)-DF\left(s,X(s)\right)\right\VV_{E^\ast}^2\ud s\ \text{.}
% \end{multline}
% Now
% \begin{align*}
%   \big\VV\sJ_n^\ast DF\left(s,\sJ_nX(s)\right)&-DF\left(s,X(s)\right)\big\VV_{E^\ast}=\sup_{\substack{e\in E\\ \VV e\VV=1}}\left\vert\langle e,\sJ_n^\ast DF\left(s,\sJ_nX(s)\right)-DF\left(s,X(s)\right)\rangle\right\vert\\
%   &=\sup_{\substack{e\in E\\ \VV e\VV =1}}\left\vert\langle \sJ_n e,DF\left(s,\sJ_n X(s)\right)\rangle-\langle e,DF\left(s,X(s)\right)\rangle\right\vert\\
%   &\leq\sup_{\substack{e\in E\\ \VV e\VV =1}}\left[\left\vert\langle\sJ_n e,DF\left(s,\sJ_n X(s)\right)\rangle-\langle \sJ_n e,DF\left(s,X(s)\right)\rangle\right\vert+\left\vert\langle \sJ_n e-e,DF\left(s,X(s)\right)\rangle\right\vert\right]\\
%   &\leq\sup_{\substack{e\in E\\ \VV e\VV =1}}\left[\left\VV\sJ_n\right\VV_E\ \left\VV DF\left(s,\sJ_nX(s)\right)-DF\left(s,X(s)\right)\right\VV_{E^\ast}+\left\VV\sJ_ne-e\right\VV_E\ \left\VV DF\left(s,X(s)\right)\right\VV_{E^\ast}\right]
% \end{align*}
and this last quantity converges to zero as before, since $\left\{\sJ_n\right\}$ is equibounded, $DF$ is continuous (hence uniformly continuous on $\left\{\sJ_nX(s)\right\}_{n,s}\cup\left\{X(s)\right\}_s$) and $\sJ_n$ converges to the identity on $E$. Since $\left\VV C(s)\right\VV^2$ is integrable, we can apply again the dominated convergence theorem in \call{eq:stochInt} to get that the left hand side converges to $0$ almost surely, hence
\begin{equation*}
  \int_0^t\langle DF\left(s,\sJ_nX(s)\right),\sJ_nC(s)\ud W(s)\rangle\to\int_0^t\langle DF\left(s,X(s)\right),C(s)\ud W(s)\rangle
\end{equation*}
in probability.\\
It remains to study the trace term. First notice that, since $E^\ast=\left( H_1\times E_2\right)^\ast\cong H_1^\ast\times E_2^\ast\cong H_1\times E_2^\ast$, every $f\in E^\ast$ can be written as a couple $\left(f_1,f_2\right)\in H_1\times E_2^\ast$ and therefore for any $u\in U$ and $f\in E^\ast$
\begin{align*}
  \tensor[_E]{\langle C(s)u,f\rangle}{_{E^\ast}}&=\tensor[_E]{\left\langle\xphi{\widetilde C(s)u}{0},\xphi{f_1}{f_2}\right\rangle}{_{E^\ast}}\\
  &={\langle \widetilde C(s)u,f_1\rangle}_1=\tensor[_U]{\langle u,\widetilde C(s)^\ast f_1\rangle}{_U}\ \text{;}
\end{align*}
hence $C(s)^\ast f=\widetilde C(s)^\ast f_1$ for any $f\in E^\ast$.\\
Now let $\sH_1$ and $\sH_2$ be complete orthonormal systems of $H_1$ and $H_2$, respectively, and set $\rH_1:=\sH_1\times\{0\}$, $\rH_2:=\sH_2\times\{0\}$, so that $\rH:=\rH_1\cup\rH_2$ is a complete orthonormal system for $H$. $\rH$ is countable since $H_1$ and $H_2$ are separable. For $h\in\rH$ we have that
\begin{equation*}
  y:=\sJ_n^\ast D^2F\left(s,\sJ_nX(s)\right)\sJ_nh\in H\subset E^\ast=H_1\times E_2^\ast,
\end{equation*}
so that, writing $y=\left(y_1,y_2\right)$, we have
\begin{equation*}
  C(s)QC(s)^\ast y=C(s)Q\widetilde C(s)^\ast y_1=\left(\begin{matrix}\widetilde C(s)Q\widetilde C(s)^\ast y_1\\ 0\end{matrix}\right)\in H_1\times\{0\}\subset E\subset H.
\end{equation*}
Therefore
\begin{equation*}
  \left\langle C(s)QC(s)^\ast\sJ_n^\ast D^2F\left(s,\sJ_nX(s)\right)\sJ_n h,h\right\rangle=\left\langle\left(\begin{matrix}\widetilde C(s)Q\widetilde C(s)^\ast y_1\\0\end{matrix}\right),h\right\rangle
\end{equation*}
and this last quantity can be different from $0$ only if $h\in\rH_1$. This implies
\begin{align}
  \nonumber \tr\big[C(s)QC(s)^\ast\sJ_n^\ast D^2F&\left(s,\sJ_nX(s)\right)\sJ_n\big]=\sum_{h\in\rH}\left\langle C(s)QC(s)^\ast\sJ_n^\ast D^2F\left(s,\sJ_nX(s)\right)\sJ_n h,h\right\rangle\\
  \nonumber &=\sum_{h\in\rH_1}{\left\langle C(s)QC(s)^\ast\sJ_n^\ast D^2F\left(s,\sJ_nX(s)\right)\sJ_n h,h\right\rangle}_1\\
  \label{eq:trH_1}&=\tr_{H_1}\left[C(s)QC(s)^\ast\sJ_n^\ast D^2F\left(s,\sJ_nX(s)\right)\sJ_n\right]\ \text{.}
\end{align}
Now, setting $\widetilde K:=\sup_n\left\Vert\sJ_n\right\Vert$ we have that for $h\in\rH_1$
\begin{align*}
    \Big\vert\tr_{H_1}&\left[C(s)QC(s)^\ast\sJ_n^\ast D^2F\left(s,\sJ_nX(s)\right)\sJ_n\right]-\tr_{H_1}\left[C(s)QC(s)^\ast D^2F\left(s,X(s)\right)\right]\Big\vert\\
    &=\left\vert\sum_{h\in\rH_1}\left\langle D^2f\left(t,\sJ_nX(s)\right)\sJ_nh,\sJ_nC(s)QC(s)^\ast h\right\rangle-\sum_{h\in\rH_1}\left\langle D^2F\left(t,X(s)\right)h,C(s)QC(s)^\ast h\right\rangle\right\vert\\
    &\leq\sum_{h\in\rH_1}\left\vert\left\langle D^2F\left(t,\sJ_nX(s)\right)\sJ_nh,\sJ_nC(s)QC(s)^\ast h-C(s)QC(s)^\ast h\right\rangle\right\vert\\
    &\phantom{\leq}+\sum_{h\in\rH_1}\left\vert\left\langle D^2F\left(t,\sJ_nX(s)\right)\sJ_nh-D^2F\left(t,X(s)\right)h,C(s)QC(s)^\ast h\right\rangle\right\vert\\
    &\leq \widetilde K\left\Vert D^2F\left(t,\sJ_nX(s)\right)\right\Vert\sum_{h\in\rH_1}\left\vert\sJ_nC(s)QC(s)^\ast h-C(s)QC(s)^\ast h\right\vert\\
    &\phantom{\leq}+\left\Vert C(s)\right\Vert_{L(U,E)}^2\ \left\Vert Q\right\Vert_{L(U,U)}^2\sum_{h\in\rH_1}\Big[\widetilde K\left\Vert D^2F\left(t,\sJ_nX(s)\right)-D^2F\left(t,X(s)\right)\right\Vert\\
      &\phantom{\leq +}+\left\Vert D^2F\left(t,X(s)\right)\right\Vert\ \left\vert\sJ_nh-h\right\vert\Big];
\end{align*}
therefore thanks to the equiboundedness of $\left\{\sJ_n\right\}$ and the uniform continuity of $D^2F$ on the set $\left\{\sJ_nX(s)\right\}_{n,s}\cup\left\{X(s)\right\}_s$ we can apply the dominated convergence theorem to the sum over $h\in\rH_1$ to obtain that
\begin{equation*}
  \tr_{H_1}\left[C(s)QC(s)^\ast\sJ_n^\ast D^2F\left(s,\sJ_nX(s)\right)\sJ_n\right]\overset{n\to\infty}{\longrightarrow}\tr_{H_1}\left[C(s)QC(s)^\ast D^2F\left(s,X(s)\right)\right]\ .
\end{equation*}
Since $D^2F$ is bounded also in $s\in[0,T]$ and $\left\VV C(s)\right\VV_{L(U;E)}^2$ is integrable by assumption, a second application of the dominated convergence theorem yields that for every $t\in[0,T]$
\begin{equation*}
  \int_0^t\tr_{H_1}\left[C(s)QC(s)^\ast\sJ_n^\ast D^2F\left(s,\sJ_nX(s)\right)\sJ_n\right]\ud s\overset{n\to\infty}{\longrightarrow}\int_0^t\tr_{H_1}\left[C(s)QC(s)^\ast D^2F\left(s,X(s)\right)\right]\ud s,
\end{equation*}
thus concluding the proof.
\end{proof}
\begin{remark}
\label{rem:Etilde}
The use of both spaces $E$ and $\widetilde E$ in the statement of the theorem can seem unjustified at first sight: since the process $X$ is supposed to live in $\widetilde E$ and the result is a It\^o formula valid on $\widetilde E$ (because the extension $G$ is defined on $\widetilde E$ only), everything could apparently be formulated in $\widetilde E$. However in most examples the space $\widetilde E$ is not a product space (see section \ref{sec:infinite_reformulation}) hence neither is its dual space, and the product structure of the dual is needed to show that the second order term is concentrated only on the $H_1$-component. Since asking $F$ to be defined on $[0,T]\times H$ will leave out many interesting examples (we typically want to endow $\widetilde E$ with a topology stronger that the one of $H$), the choice to use the intermediate space $E$ seems to be the more adequate.
\end{remark}
\begin{corollary}
\label{coro:piecewise}
  Consider $n+1$ points $0=t_0\leq t_1\leq\dots\leq t_n= T$ and assume that $F\in C\left([t_j,t_{j+1})\times E;\bR\right)$ for $j=0,1,\dots,n-1$. Suppose moreover that
  \begin{enumerate}
  \item the map $y\mapsto F(t,y)$ is twice differentiable for every $t\in[0,T]$;
  \item $DF\in C\left(\left[t_j,t_{j+1}\right)\times E;E^\ast\right)$ for $j=0,1,\dots,n-1$;
  \item $D^2F\in C\left(\left[t_j,t_{j+1}\right)\times E;L\left(E,E^\ast\right)\right)$ for $j=0,1,\dots,n-1$;
  \item the map $t\mapsto F(t,y)$ is càdlàg for every $y\in E$;
  \item $\frac{\partial F}{\partial t}$ exists for $(t,x)\in \sT\times\widetilde D$ where $\sT$ is as in Theorem \ref{thm:Ito_Banach};
  \item there exists a function $G$ such that $G\in C\left(\left[t_j,t_{j+1}\right)\times\widetilde E;\bR\right)$ for all $j$ and such that
    \begin{equation*}
      G(t,x)=\frac{\partial F}{\partial t}(t,x)+\langle Ax,DF(t,x)\rangle\quad \forall x\in\widetilde D,\ \forall t\in\sT\cap\left[t_j,t_{j+1}\right)\ .
    \end{equation*}
  \end{enumerate}
Then the formula
\begin{align*}
  F\left(T,X(T)\right)&=F\left(0,X^0\right)+\sum_{j=1}^n\left[F\left(t_j,X\left(t_j\right)\right)-F\left(t_j-,X\left(t_j\right)\right)\right]\\
  &+\int_0^TG\left(s,X(s)\right)\ud s+\int_0^T\langle DF\left(s,X(s)\right),C(s)\ud W(s)\rangle\\
  &+\int_0^T\left(\langle B(s),DF\left(s,X(s)\right)\rangle+\frac{1}{2}\tr_{H_1}\left[C(s)QC(s)^\ast D^2F\left(s,X(s)\right)\right]\right)\ud s
\end{align*}
holds.
\end{corollary}
\begin{proof}
Thanks to the assumptions, theorem \ref{thm:Ito_Banach} can be applied to obtain $n$ identities for the increments $F\big(t_{j+1}-\epsilon,X\left(t_{j+1}-\epsilon\right)\big)-F\big(t_j,X\left(t_j\right)\big)$, $j=0,\dots,n-1$, with $0<\epsilon<\min_j\left(t_{j+1}-t_j\right)$. Summing up these identities and taking the limit as $\epsilon$ goes to $0$ yields the result.
%  Choose $\epsilon$ such that $\min_j\left(t_{j+1}-t_j\right)>\epsilon >0$ and
\end{proof}
\section{Application to generators of groups}

In a Hilbert space $H$, given a Wiener process $\left(  W\left(  t\right)
\right)  _{t\geq0}$ with covariance $Q$, defined on a filtered probability
space $\left(  \Omega,\mathcal{F},\left(  \mathcal{F}_{t}\right)  _{t\geq
0},P\right)  $, given $x^{0}\in H$, $B:\Omega\times\left[  0,T\right]
\rightarrow H$ progressively measurable and integrable in $t$, $P$-a.s.,
$C:\Omega\times\left[  0,T\right]  \rightarrow L\left(  H,H\right)  $
progressively measurable and square integrable in $t$, $P$-a.s., let $X\left(
t\right)  $ be the stochastic process given by the mild formula%
\[
X\left(  t\right)  =e^{tA}x^{0}+\int_{0}^{t}e^{\left(  t-s\right)  A}B\left(
s\right)  ds+\int_{0}^{t}e^{\left(  t-s\right)  A}C\left(  s\right)  dW\left(
s\right)
\],
where $e^{tA}$ is a strongly continuous \textit{group}. In this particular
case we can also write%
\[
X\left(  t\right)  =e^{tA}\left(  x^{0}+\int_{0}^{t}e^{-sA}B\left(  s\right)
ds+\int_{0}^{t}e^{-sA}C\left(  s\right)  dW\left(  s\right)  \right)
\],
from which we may deduce, for instance, that $X$ is a continuous process in
$H$. Formally%
\[
dX\left(  t\right)  =AX\left(  t\right)  dt+B\left(  t\right)  dt+C\left(
t\right)  dW\left(  t\right),
\]
but $AX\left(  t\right)  $ is generally not well defined:\ typically the
solution has the same spatial regularity of the initial condition and the
forcing terms. Thus in general, one cannot apply the classical It\^{o} formula
to $F\left(  t,X\left(  t\right)  \right)  $, due to this fact. A possibility
is given by the mild It\^{o} formula \cite{DPJR}. We show here an
alternative, which applies when suitable cancellations in $F\left(
t,x\right)  $ occur. 

As a first example, let $F\left(  t,x\right)  $ be given by
\[
F\left(  t,x\right)  =F_{0}\left(  e^{-tA}x\right)  +\int_{0}^{t}H_{0}\left(
s,e^{-\left(  t-s\right)  A}x\right)  ds,
\]
where $F_{0}\in C^{2}\left(  H;\mathds{R}\right)  $, $H_{0}\in C\left(
\left[  0,T\right]  \times H;\mathds{R}\right)  $, with continuous derivatives
$DH_{0}$, $D^{2}H_{0}$. Then $\frac{\partial F}{\partial t}\left(  t,x\right)
$ exists for all $x\in D\left(  A\right)  $, $t\in\left[  0,T\right]  $ and it
is given by%
\[
\frac{\partial F}{\partial t}\left(  t,x\right)  =-\left\langle \left(
DF_{0}\right)  \left(  e^{-tA}x\right)  ,e^{-tA}Ax\right\rangle +H_{0}\left(
t,x\right)  -\int_{0}^{t}\left\langle \left(  DH_{0}\right)  \left(
s,e^{-\left(  t-s\right)  A}x\right)  ,e^{-\left(  t-s\right)  A}%
Ax\right\rangle ds.
\]
Moreover, $DF\in C\left(  \left[  0,T\right]  \times H;H\right)  $, $D^{2}F\in
C\left(  \left[  0,T\right]  \times H;L\left(  H,H\right)  \right)  $ and
\[
\left\langle DF\left(  t,x\right)  ,h\right\rangle =\left\langle \left(
DF_{0}\right)  \left(  e^{-tA}x\right)  ,e^{-tA}h\right\rangle +\int_{0}%
^{t}\left\langle DH_{0}\left(  s,e^{-\left(  t-s\right)  A}x\right)
,e^{-\left(  t-s\right)  A}h\right\rangle ds.
\]
Therefore
\[
\frac{\partial F}{\partial t}\left(  t,x\right)  +\left\langle Ax,DF\left(
t,x\right)  \right\rangle =H_{0}\left(  t,x\right)  .
\]
Consider the function $G\left(  t,x\right)  :=\frac{\partial F}{\partial
t}\left(  t,x\right)  +\left\langle Ax,DF\left(  t,x\right)  \right\rangle $.
It is a priori well defined only on $x\in D\left(  A\right)  $. However,
being
\[
G\left(  t,x\right)  =H_{0}\left(  t,x\right),
\]
the function $G$ extends to a continuous function on $\left[  0,T\right]
\times H$. Then theorem \ref{thm:Ito_Hilbert} applies and It\^{o} formula reads
\begin{align*}
&  F\left(  t,X\left(  t\right)  \right)  =F\left(  0,x^{0}\right)  +\int%
_{0}^{t}H_{0}\left(  s,X\left(  s\right)  \right)  ds+\int_{0}^{t}\left\langle
B\left(  s\right)  ,DF\left(  s,X\left(  s\right)  \right)  \right\rangle ds\\
&  +\int_{0}^{t}\left\langle DF\left(  s,X\left(  s\right)  \right)  ,C\left(
s\right)  dW\left(  s\right)  \right\rangle +\frac{1}{2}\int_{0}^{t}Tr\left(
C\left(  s\right)  QC^{\ast}\left(  s\right)  D^{2}F\left(  s,X\left(
s\right)  \right)  \right)  ds.
\end{align*}

\subsection{Kolmogorov equation for SDEs with group generator}

The previous example concerns a very particular class of functionals $F$. As a
more useful (but very related)\ example, assume we have a solution $F\left(
t,x\right)  $ of the following Kolmogorov equation%
\begin{equation}
\frac{\partial F}{\partial t}\left(  t,x\right)  +\left\langle Ax+B\left(
t,x\right)  ,DF\left(  t,x\right)  \right\rangle +\frac{1}{2}Tr\left(
C\left(  t,x\right)  QC^{\ast}\left(  t,x\right)  D^{2}F\left(  t,X\left(
t\right)  \right)  \right)=0,   \label{Kolmog group}
\end{equation}
for $x\in D\left(  A\right)  ,t\in\left[0,T\right]$, $F\left(  T,x\right)=\phi\left(  x\right)$, with the regularity
\begin{align}
F  & \in C\left(  \left[  0,T\right]  \times H;\mathds{R}\right)  ,\qquad
DF\in C\left(  \left[  0,T\right]  \times H;H\right)  \label{class}\\
D^{2}F  & \in C\left(  \left[  0,T\right]  \times H;L\left(  H,H\right)
\right)  ,\qquad\frac{\partial F}{\partial t}\in C\left(  \left[  0,T\right]
\times D\left(  A\right)  ;\mathds{R}\right)  .\nonumber
\end{align}
Here we assume that $B:\left[  0,T\right]  \times H\rightarrow H$ and
$C:\left[  0,T\right]  \times H\rightarrow L\left(  H,H\right)  $ are
continuous (we assume continuity of $B$ and $\frac{\partial F}{\partial t}$
for simplicity of exposition, but this detail can be generalized). Since%
\[
G\left(  t,x\right)  :=\frac{\partial F}{\partial t}\left(  t,x\right)
+\left\langle Ax+B\left(  t,x\right)  ,DF\left(  t,x\right)  \right\rangle
,\qquad x\in D\left(  A\right)  ,t\in\left[  0,T\right],
\]
satisfies
\[
G\left(  t,x\right)  =-\frac{1}{2}Tr\left(  C\left(  t,x\right)  QC^{\ast
}\left(  t,x\right)  D^{2}F\left(  t,X\left(  t\right)  \right)  \right),
\]
then it has a continuous extension on $\left[  0,T\right]  \times H$ and
theorem \ref{thm:Ito_Hilbert} is
applicable, if $\left(  X\left(  t\right)  \right)  _{t\in\left[
t_{0},T\right]  }$ (for some $t_{0}\in\lbrack0,T)$) is a continuous process in
$H$ satisfying%
\begin{equation}
X\left(  t\right)  =e^{\left(  t-t_{0}\right)  A}x^{0}+\int_{t_{0}}%
^{t}e^{\left(  t-s\right)  A}B\left(  s,X\left(  s\right)  \right)
ds+\int_{t_{0}}^{t}e^{\left(  t-s\right)  A}C\left(  s,X\left(  s\right)
\right)  dW\left(  s\right)  \ \text{,}\label{SDE group}%
\end{equation}
thus we get%
\begin{equation}
\label{eq:kolmgroup1}
F\left(  t,X\left(  t\right)  \right)  =F\left(  t_{0},x^{0}\right)
+\int_{t_{0}}^{t}\left\langle DF\left(  s,X\left(  s\right)  \right)
,C\left(  s,X\left(  s\right)  \right)  dW\left(  s\right)  \right\rangle \ .
\end{equation}
% and this identity implies, when for instance $DF$ and $C$ are bounded, that
% \[
% F\left(  t_{0},x^{0}\right)  =E\left[  F\left(  t,X\left(  t\right)  \right)
% \right]  .
% \]
% The same result holds if $F$ is bounded, since then, thanks to \call{eq:kolmgroup1}, $\int_{t_{0}}^{t}\left\langle DF\left(  s,X\left(  s\right)  \right)  ,C\left(  s,X\left(s\right)  \right)  dW\left(  s\right)  \right\rangle $ is a uniformly integrable local martingale, hence a martingale . 
We can now easily prove the following uniqueness result. We do not repeat the assumptions on $H$, $W$, $e^{tA}$, $B$.

\begin{theorem}
Assume that for every $\left(  t_{0},x^{0}\right)  \in\left[  0,T\right]
\times H$, there exists at least one continuous process $X$ in $H$ satisfying
equation (\ref{SDE group}). Then the following holds.

i) The Kolmogorov equation (\ref{Kolmog group}) has a unique solution in the
class of bounded functions $F$ satisfying (\ref{class}).

ii)\ If $C\in C_{b}\left(  \left[  0,T\right]  \times H;L\left(  H,H\right)
\right)  $, it has a unique solution in the class of functions $F$ satisfying
(\ref{class}) and $\left\Vert DF\right\Vert _{\infty}<\infty$. 
\end{theorem}
\begin{proof}
 % Proof.
 i) The stochastic integral $\int_{t_{0}}^{t}\left\langle DF\left(
s,X\left(  s\right)  \right)  ,C\left(  s,X\left(  s\right)  \right)
dW\left(  s\right)  \right\rangle $ is a local martingale. Since $F$ is
bounded, from identity \call{eq:kolmgroup1} it follows that it is uniformly integrable, hence
it is also a martingale. Therefore, taking expectation in \call{eq:kolmgroup1}, we get
$F\left(  t_{0},x^{0}\right)  =E\left[  \varphi\left(  X\left(  T\right)
\right)  \right]  $, formula which identifies $F$, since $t_{0}$ and $x^{0}$
are arbitrary.

ii) If $C$ and $DF$ are bounded, the the stochastic integral in \call{eq:kolmgroup1} is a
martingale. We conclude as in i). 
\end{proof}

\section{Application to path-dependent
functionals\label{sec:infinite_reformulation}}
We will now apply the abstract results of section \ref{sec:ItoBanach} to obtain an It\^o formula for path-dependent functionals of continuous processes, stated in theorem \ref{thm:ItoPath}. In the first sections we will introduce the necessary spaces and operators and we will show that the infinite-dimensional reformulation of path-dependent problems appears naturally when dealing with path-dependent SDEs (see again \cite{FZ15} for a more detailed discussion).
\subsection{Infinite-dimensional formulation of It\^o processes}
\label{subsec:infinite_reformulation_proc}
In this and the following sections we will denote by $C_t$ the space of $\bR^d$-valued functions on $[0,t]$ that can have a jump only at $t$, that is
\begin{equation*}
  C_t=\left\{\phi:[0,t]\to\bR^d\colon\phi\in C\left([0,t);\bR^d\right),\ \exists \lim_{s\to t^-} \phi(s)\in\bR^d\right\}
\end{equation*}
endowed with the supremum norm. The space $C_t$ is clearly isomorphic to the product space
\begin{equation*}
  \bR^d\times\left\{\phi\in C\left([0,t);\bR^d\right),\ \exists \lim_{s\to t^-} \phi(s)\in\bR^d\right\}\ .
\end{equation*}

Let $y(t)$ be the continuous process in $\bR^d$ given by
\begin{equation}
  \label{eq:ItoRd}
  y(t)= y^0+\int_0^tb(s)\ud s+\int_0^tc(s)\ud W(s)\ ,
\end{equation}
where $b$ and $c$ are progressively measurable processes, with values in $\bR^d$ and $\bR^{k\times d}$ respectively, such that
\begin{equation*}
  \int_0^T\left\vert b(s)\right\vert\ud s<\infty\ ,\qquad \int_0^T\left\Vert c(s)\right\Vert^2\ud s<\infty
\end{equation*}
and $y^0$ is a $\sF_0$-measurable random vector.\\

Let us introduce the following infinite-dimensional reformulation. %Denote by $D_b\left([-T,0);\bR^d\right)$ the space of càdlàg functions $\phi(t)$ from $[-T,0)$ to $\bR^d$ that have finite left limit also for $t\to 0^-$. 
We will work in the space
 % \begin{equation*}
 %   O=\bR^d\times \left\{\phi\in D\left([-T,0);\bR^d\right)\colon \exists\lim_{s\to 0^-}\phi(s)\in\bR^d\right\}\ \text{,}
 % \end{equation*}
% where $D_b\left([-T,0);\bR^d\right)$ denotes the space of càdlàg functions $\phi(t)$ that have finite left limit also for $t\to 0$ endowed with the same norm as $E$. This norm turns $O$ into a non-separable Banach space, and we have the inclusions
% \begin{equation*}
%   \widetilde E\subset E\subset O\subset H\ .
% \end{equation*}
\begin{equation}
\label{eq:Epathdep}
E=\mathds{R}^{d}\times \left\{\phi\in C\left([-T,0);\bR^d\right)\colon \exists\lim_{s\to 0^-}\phi(s)\in\bR^d\right\}\ \text{,}
\end{equation}
whose elements we shall usually denote by $x=\xphi{x_1}{x_2}$. $E$ is a Banach space when endowed with the norm $\left\Vert\xphi{x_1}{x_2}\right\Vert^2=\left\vert x_1\right\vert^2+\left\Vert x_2\right\Vert^2_\infty$; the notation $\left\langle \cdot,\cdot\right\rangle $ will denote the duality pairing between $E$ and its dual space $E^\ast$.
%\added[id=GZ]{The reason to chose the space od càdlàg paths instead of the more natural space of continuous paths is mainly technical and will be clearer in the following subsection.}
The space $E$ is densely and continuously embedded in the product space
\begin{equation}
  \label{eq:Hpathdep}
H=\bR^d\times L^2\left(-T,0;\bR^d\right)\ \text{.}
\end{equation}
%hence also is $O$; $E$ is not dense in $O$.
%(the solution of the differentialequation at time $t$ will be $X\left(  t\right)  =\left(  X_{1},X_{2}\right)^{T}$). 
We also introduce the unbounded linear operator $A:D\left(
A\right)  \subset H\rightarrow H$ defined as%
% \begin{equation}
% \label{eq:DAE}
% D\left(A_E\right)=\left\{\xphi{x_1}{x_2}\in E \colon x_2\in C^1\left([-T,0);\bR^d\right),\ x_1=\lim_{t\to 0^-}x_2(t)\right\}\ \text{,}
% \end{equation}
\begin{equation}
\label{eq:DAE}
D\left(A\right)=\left\{\xphi{x_1}{x_2}\in H \colon x_2\in W^{1,2}\left(-T,0;\bR^d\right),\ x_1=\lim_{s\to 0^-}x_2(s)\right\}\ \text{,}
\end{equation}
\begin{equation}
\label{eq:A}
A=\left(
\begin{array}
[c]{cc}%
0 & 0\\
0 & \frac{d}{dr}%
\end{array}
\right), 
\end{equation}
where we identify an element in $W^{1,2}$ with the restriction of its continuous version to $[-T,0)$.
Therefore we identify also the space
 \begin{equation}
 \label{eq:Etildepathdep}
 \widetilde E=\overline{D\left(A\right)}^E=\left\{y=\xphi{x_1}{x_2}\in E \colon x_1=\lim_{s\to 0^-}x_2(s)\right\}\ .
 \end{equation}
%which is isomorphic to $C\left([-T,0];\bR^d\right)$.
% The same operator $A$ is defined as well on $H$ (and denoted by $A_{H}$ there) with domain
% \begin{equation}
%   \label{eq:DAH}
%   D\left(A_{H}\right)=\left\{\xphi{x_1}{x_2}\in\bR^d\times W^{1,2}\left(-T,0;\bR^d\right)\colon x_2(0)=x_1\right\}
% \end{equation}
The operator $A$ generates a strongly continuous semigroup $e^{tA}$ in $H$. This semigroup turns out to be not strongly continuous in $E$; nevertheless $e^{tA}$ maps $\widetilde E$ in itself and is strongly continuous in $\widetilde E$. This follows from the fact that the semigroup $e^{tA}$ has the explicit form
\begin{equation}
\label{eq:etA}
e^{tA}x=\left(
\begin{array}
[c]{c}%
x_{1}\\
x_{2}\left(  \cdot+t\right)  \ind_{\left[  -T,-t\right)  }+x_{1}\ind_{\left[
-t,0\right)  }%
\end{array}
\right)
\end{equation}
(see \cite{BDPDM} for details on the operator $A$ in the context of delay equations and \cite{FZ15} about its role in the theory of path-dependent equations).\\
For any $t\in[0,T]$ we introduce the operator
\begin{equation*}
L^{t}:C_t  \rightarrow E,
\end{equation*}
defined as
\begin{equation}
\label{eq:Lt}
L^{t}\gamma  =\left(\begin{matrix}\gamma(t) \\ \gamma(0)\ind_{[-T,-t)}+\gamma(t+\cdot)\ind_{[-t,0)}\end{matrix}\right)\ 
\end{equation}
for every $\gamma\in C_t  $.\\
Using \call{eq:etA}, it is easy to show (see also proposition \ref{prop:reformulation} below) that 
\begin{equation*}
  X(t)=L^ty_t\ ,
  \end{equation*}
as a $H$-valued process, is given by
\begin{equation}
  \label{eq:procH}
  X(t)=e^{tA}X^0+\int_0^t\etsa B(s)\ud s+\int_0^t\etsa C(s)\ud W(s),
\end{equation}
where $X^0=\xphi{x^0}{x^0\ind_{[-T,0)}}$ and the processes $B:[0,T]\to E$ and $C:[0,T]\to L\left(\bR^k,E\right)$ are given by
\begin{equation}
  \label{eq:BC}
  B(t)=\left(\begin{matrix}b(s)\\0\end{matrix}\right)\ ,\qquad C(s)u=\left(\begin{matrix}c(s)u\\0\end{matrix}\right)\;\text{for }u\in \bR^k\ .
\end{equation}
The validity of \call{eq:procH} corresponds to saying that $X$ is the unique mild solution to the linear equation
\begin{equation}
  \label{eq:procSDE}
  \ud X(t)=AX(t)\ud t+B(t)\ud t+C(t)\ud W(t)\ \text{;}
\end{equation}
hence we see that our infinite-dimensional reformulation forces us to deal with \emph{equations} even if we start from finite-dimensional processes: the operator $A$ appears as a consequence of the introduction of the second component that represents the ``past trajectory'' of the process (see remark \ref{rem:eqnA}).\\
\begin{proposition}
\label{prop:continuity}
The process $X$ is such that $X(t)\in\widetilde E$ for every $t$ and the trajectories $t\mapsto X(t)$ are almost surely continuous as maps from $[0,T]$ to $\widetilde E$.
\end{proposition}
\begin{proof}
The random variable $X^0$ takes values in $\widetilde E$ by definition. Since the process $y$ has almost surely continuous trajectories, $\left(L^ty_t\right)_2\in C\left([-T,0);\bR^d\right)$ and $L^ty_t$ belongs to $\widetilde E$. To check the almost sure continuity of the trajectories of $X$ as a $\widetilde E$-valued process denote again by $\Omega_0\subset\Omega$ a null set such that $t\mapsto y(\omega,t)$ is continuous for every $\omega\in\Omega\setminus\Omega_0$, fix $\omega\in\Omega\setminus\Omega_0$, fix $t,s\in[0,T]$ and $\epsilon>0$; we can suppose $t>s$ without loss of generality. Since $y\left(\omega,\cdot\right)$ is uniformly continuous on $[0,T]$ we can find $\delta$ such that $\left\vert y(t)-y(s)\right\vert<\nicefrac{\epsilon}{2}$ if $t-s<\delta$. Then for $t-s<\delta$
\begin{align*}
  \tensor{\left\Vert X(t)-X(s)\right\Vert}{_{\widetilde E}}&\leq\left\vert y(t)-y(s)\right\vert+\max\left\{\sup_{r\in[0,t-s]}\left\vert y(0)-y(r)\right\vert,\sup_{r\in[0,s]}\left\vert y(t-s+r)-y(r)\right\vert\right\}\\
  &\leq \epsilon\ .
  \end{align*}
\end{proof}

\subsection{Infinite-dimensional formulation of path-dependent functionals}
\label{subsec:functionals}

A path-dependent functional $f$ is a family of functionals $f\left(
t,\cdot\right)  $, $t\in\left[  0,T\right]  $, such that 
$$ f\left(t,\cdot\right):C_t\to\bR\ .$$
Among the examples of path-dependent functional, let us mention the integral ones%
\begin{equation}
f\left(  t,\gamma\right)  =\int_{0}^{t}g\left(\gamma(t),  \gamma\left(  s\right)
\right)  ds \label{example 1}%
\end{equation}
and those involving pointwise evaluations, like for instance%
\begin{equation}
f\left(  t,\gamma\right)  =q\left(  \gamma(t),\gamma(t_{0})\right)
\ind_{t> t_{0}}. \label{example 2}%
\end{equation}
Here we assume that $g:\mathds{R}^{d}\times\bR^d\rightarrow\mathds{R}$ is a measurable
function in the first example, with $\left\vert g\left(  a\right)
\right\vert \leq C\left(  1+\left\vert a\right\vert ^{2}\right)  $ and
that\ $q:\mathds{R}^{d}\times\mathds{R}^{d}\rightarrow\mathds{R}$ is a
measurable function in the second example and $t_{0}\in\left[  0,T\right]  $
is a given point. In order to apply It\^{o} calculus let us simplify and
assume that $g$ and $q$ are twice continuously differentiable with
bounded derivatives of all orders.

Given a path-dependent functional $f\left(  t,\cdot\right)  $, $t\in\left[
0,T\right]  $ we may associate to it a map
$F:\left[  0,T\right]  \times E\rightarrow\mathds{R}$ setting
\begin{equation}
\label{eq:F}
F\left(  t,x\right)  =f\left(  t,M_{t}x\right) 
\end{equation}
where
\begin{equation*}
  M_{t}:E\to C_t
\end{equation*}
is defined as
\begin{equation}
\label{eq:defM}
  M_{t}x  \left(  s\right)  =x_{2}\left(  s-t\right)\ind_{[0,t)}(s)+x_1\ind_{\{t\}}(s)  ,\qquad
s\in\left[  0,t\right]\  .
\end{equation}
\begin{remark}
\label{rem:ML}
Notice that $M_{t}L^{t}$ is the identity on $C_t $.
%Moreover the operator $M_t$ can be defined through \call{eq:defM} also as an operator from $H$ to $L^2\left(0,t;\bR^d\right)$; this is not possible with $L^t$, therefore their composition is well defined only on $C_t$.
\end{remark}
Here we see that if $f$ were defined only on $C\left([0,t];\bR^d\right)$, then $F$ would be defined on $[0,T]\times \widetilde E$, because of the definition of the operators $M_t$; our abstract results require instead $F$ to be defined on $[0,T]\times E$, see remark \ref{rem:Etilde}. \\
Path-dependent functionals are often studied in spaces of càdlàg paths. The framework presented here can be easily modified to do so, similarly to what is done in \cite{FZ15}; however this would require the introduction of further spaces, thus complicating notations, but would not lead to generalizations of the result proved here.

The aim of this section is to show that examples (\ref{example 1}) and (\ref{example 2}) fulfill the assumptions of theorem \ref{thm:Ito_Banach}.

The abstract reformulation of the functional given in (\ref{example 1}) is the map
$F:\left[  0,T\right]  \times E\rightarrow\mathds{R}$ defined as%
\begin{align*}
F\left(  t,x\right)  =\int_{0}^{t}g\left(M_tx(t),  M_{t}x
\left(  s\right)  \right)  ds&=\int_{0}^{t}g\left(x_1,  x_{2}\left(  s-t\right)\right)  ds\\ &=\int_{-t}^{0}g\left(x_1,  x_{2}\left(  r\right)\right)  dr.
\end{align*}
Hence%
\begin{align*}
\frac{\partial F}{\partial t}\left(  t,x\right)   &  =g\left(x_1,  x_1\right)  -\int_{0}^{t}Dg\left(x_1,  x_{2}\left(  s-t\right)  \right)
\cdot x_{2}^{\prime}\left(  s-t\right)  ds\\
&  =g\left(x_1,x_1\right)-\int_{-t}^0g\left(x_1,  x_{2}\left(r\right)  \right)\cdot x_2^\prime\left(r\right)\ud r,   
\end{align*}
which is meaningful for example if $x_{2}$ belongs to $C^1\left([-T,0];\bR^d\right)$. Indeed since $A\big(D(A)\big)\subset\{0\}\times L^2\left(-T,0;\bR^d\right)$, we have 
\begin{equation*}
  \widetilde D=A^{-1}(E)=A^{-1}\left(\{0\}\times E_2\right),
\end{equation*}
where
\begin{equation*}
  E_2=\left\{\phi\in C\left([-T,0);\bR^d\right)\colon \exists\lim_{s\to 0^-}\phi(s)\in\bR^d\right\}\ \text{,}
\end{equation*}
and so
\begin{equation*}
  \widetilde D=\left\{\xphi{x_1}{x_2}\in D(A)\colon x_2\in C^1\left([-T,0);\bR^d\right)\right\}\ .
\end{equation*}
% more regular than $L^{2}\left(
% -T,0;\mathds{R}^{d}\right)  $, for instance if $x_{2}\in W^{1,2}\left(
% -T,0;\mathds{R}^{d}\right)  $. Notice that under such condition, by Sobolev
% embedding theorem, $x_{2}$ is also continuous and thus the pointwise value
% $g\left(  x_{2}\left(
% -t\right)  \right)  $ is well defined; 
Moreover, the time derivative of $F$ is defined for \emph{every} $t\in[0,T]$. Therefore we see that
\[
\frac{\partial F}{\partial t}:\left[  0,T\right]  \times \widetilde D\to \bR
\]
is a natural assumption, while $\frac{\partial F}{\partial t}:
\left[  0,T\right]  \times E\to\bR  $ would not be. Since $g$ is continuous we also have that $\partial_t F$ belongs to $C\left([0,T]\times \widetilde D;\bR\right)$.\\
Let us then investigate the function%
\[
G\left(  t,x\right)  :=\frac{\partial F}{\partial t}\left(  t,x\right)
+\left\langle Ax,DF\left(  t,x\right)  \right\rangle \qquad x\in \widetilde D,t\in\left[  0,T\right]  .
\]
% It is given by%
% \[
% G\left(  t,x\right)  =g\left(  x_{2}\left(  -t\right)  \right)  +\int%
% _{-t}^{0}Dg\left(  x_{2}\left(  r\right)  \right)  \cdot x_{2}^{\prime
% }\left(  r\right)  dr
% \]
For $h\in E$ we have
\begin{align*}
\left\langle h,DF\left(  s,x\right)  \right\rangle  &  =\lim_{\epsilon
\rightarrow0}\frac{1}{\epsilon}\int_{0}^{t}\left(  g\left(x_1,  \left(
x_{2}+\epsilon h_{2}\right)  \left(  s-t\right)  \right)  -g\left(x_1,
x_{2}\left(  s-t\right)  \right)  \right)  ds\\
&  =\int_{0}^{t}Dg\left(x_1,  x_{2}\left(  s-t\right)  \right)  \cdot
h_{2}\left(  s-t\right)  ds=\int_{-t}^{0}Dg\left(x_1,  x_{2}\left(  r\right)
\right)  \cdot h_{2}\left(  r\right)  dr.
\end{align*}
But then
\begin{equation*}
  G\left(s,x\right)=g\left(x_1,x_1\right);
\end{equation*}
thus we see that the function $G\left(  s,x\right)  $ is well defined on $E$ too! The assumption $G\in C\left(  \left[  0,T\right]  \times \widetilde E\right)  $ is fulfilled.\\
The abstract reformulation of the functional given in (\ref{example 2}) is the map
$F:E\rightarrow\mathds{R}$ defined as%
\begin{align*}
F\left(  t,x\right)    &=q\left(  M_{t}x(t),M_{t}x(t_{0})\right)
\ind_{t> t_{0}}=q\left(  x_{1},x_{2}\left(  t_{0}-t\right)  \right)  \ind_{t>t_{0}}\ .
\end{align*}
Hence, writing $\partial_1q$ and $\partial_2q$ for the derivatives of $q$ with respect to its first and second variable, respectively, for $t\neq t_{0}$,
\begin{equation*}
\frac{\partial F}{\partial t}\left(  t,x\right)  =-\partial_2q\left(
x_{1},x_{2}\left(  t_{0}-t\right)  \right)  \cdot x_{2}^{\prime}\left(
t_{0}-t\right)  \ind_{t> t_{0}},
\end{equation*}
which requires $x_2\in C^1$. 

% \[
% D\left(  A_{E}\right)  =\left\{  x\in\mathds{R}^{d}\times D^{1}\left(
% -T,0;\mathds{R}^{d}\right)  :x_{2}\left(  0\right)  =x_{1}\right\}  .
% \]
But%
\begin{align*}
G\left(  t,x\right)   &  =-\partial_{2}q\left(  x_{1},x_{2}\left(  t_{0}-t\right)
\right)  \cdot x_{2}^{\prime}\left(  t_{0}-t\right)  \ind_{t>t_{0}}\\
&  +\partial_{1}q\left(  x_{1},x_{2}\left(  t_{0}-t\right)  \right)  \ind_{t>t_{0}%
}\cdot\left(  Ax\right)  _{1}+\partial_{2}q\left(  x_{1},x_{2}\left(
t_{0}-t\right)  \right)  \ind_{t>t_{0}}\cdot\left(  Ax\right)  _{2}\left(
t_{0}-t\right)  \\
&  =0,
\end{align*}
because $\left(  Ax\right)  _{1}=0$ and $\left(  Ax\right)  _{2}\left(
t_{0}-t\right)  =x_{2}^{\prime}\left(  t_{0}-t\right)  $. Again, $G$ extends
continuously to $E$.

\subsection{infinite-dimensional formulation of path-dependent SDEs}
\label{subsec:infinite_reformulation_SDE}

% Similarly to what we did in subsection \ref{subsec:infinite_reformulation_SDE} we briefly show how to formulate a path-dependent SDE in our infinite-dimensional \todonum{qui manca un remark che spieghi che nelle SDE la necessità della costruzione si vede meglio}framework.\\
In subsection \ref{subsec:infinite_reformulation_proc} we have formulated a classical It\^{o} process as an
infinite-dimensional process given by a mild formula; this apparently not
natural formulation is suggested by the case when the process is the solution
of a path-dependent SDE. For these equations, the mild formulation is natural,
due to the similarity with delay equations, where the infinite-dimensional
approach is classical. Let us give here some details about the case of a
path-dependent SDE.

Let $\left(  \Omega,\mathcal{F},\bP\right)  $ be a complete probability space,
$\mathbb{F}=\left(  \mathcal{F}_{t}\right)  _{t\geq0}$ a complete filtration,
$\left(  W\left(  t\right)  \right)  _{t\geq0}$ a Brownian motion in
$\mathds{R}^{k}$ (we shall write $W^{i}\left(  t\right)  $ for its
coordinates), $y^{0}$ an $\mathcal{F}_{0}$-measurable random vector of
$\mathds{R}^{d}$. Consider the path-dependent SDE in $\mathds{R}^{d}$%
\[
dy\left(  t\right)  =b\left(  t,y_{t}\right)  dt+\sigma\left(  t,y_{t}\right)
dW\left(  t\right)  ,\qquad y\left(  0\right)  =y^{0}.
\]
The solution $\left(  y\left(  t\right)  \right)  _{t\in\left[  0,T\right]  }$
is a stochastic process in $\mathds{R}^{d}$ and $y_{t}$
denotes the window%
\[
y_{t}:=\left\{  y\left(  s\right)  \right\}  _{s\in\left[  0,t\right]  }.
\]
About $b$ and $\sigma$, initially  we assume that, for each
$t\in\left[  0,T\right]  $, the function $b\left(  t,\cdot\right)  $ maps $C\left(  [0,t];\mathds{R}^{d}\right)  $ into $\mathds{R}^{d}$ and the
function $\sigma\left(  t,\cdot\right)  $ maps $C\left(  [0,t];\mathds{R}%
^{d}\right)  $ into $\mathds{R}^{k\times d}$; moreover, we assume that $b$ ad $\sigma$ are
locally bounded measurable functions, for each $t\in\left[
0,T\right]  $, with bounds uniform in $t$ and that the processes $b\left(
t,y_{t}\right)  $ and $\sigma\left(  t,y_{t}\right)  $ are progressively
measurable. These are relatively
weak requirements to give a meaning to the integral equation%
\begin{equation}
y\left(  t\right)  =y^{0}+\int_{0}^{t}b\left(  s,y_{s}\right)  ds+\int_{0}%
^{t}\sigma\left(  s,y_{s}\right)  dW\left(  s\right)  .\label{path dep SDE}%
\end{equation}
If, in addition, we assume that $b\left(  t,\cdot\right)  $ and $\sigma\left(
t,\cdot\right)  $ are Lipschitz continuous from $C\left(  [0,t];\mathds{R}%
^{d}\right)  $ to $\mathds{R}^{d}$ and $\mathds{R}^{k\times d}$ respectively,
with Lipschitz constants independent of $t$, then existence and uniqueness of a continuous
solution, adapted to the completed filtration of $W$, holds true. We shall
also write $\sigma\left(  t,y_{t}\right)  dW\left(  t\right)  $ as $\sum
_{i=1}^{k}\sigma_{i}\left(  t,y_{t}\right)  dW^{i}\left(  t\right)  $.

We take the operator $A$ as in section \ref{subsec:infinite_reformulation_proc} and we define the continuous nonlinear operators $B:\left[  0,T\right]  \times \widetilde E\rightarrow E$, $C_{i}:\left[  0,T\right]  \times \widetilde E\rightarrow E$, $i=1,\dots,k$, as%
\begin{equation}
\label{eq:B}
B\left(  t,x\right)  =\left(
\begin{array}
[c]{c}%
b\left(  t,M_{t}x\right)  \\
0
\end{array}
\right)
\end{equation}

\begin{equation}
\label{eq:C}
C_{i}\left(  t,x\right)  =\left(
\begin{array}
[c]{c}%
\sigma_{i}\left(  t,M_{t}x\right)  \\
0
\end{array}
\right)\ .
\end{equation}
Finally, we set $U=\mathds{R}^{k}$, take $Q$ equal to the identity in $U$ and
consider, for every $\left(  t,x\right)  \in\left[  0,T\right]  \times E$, the
bounded linear operator $C\left(  t,x\right)  :U\rightarrow E$ having
components $C_{i}\left(  t,x\right)  $.\\
Given a $\sF_0$-measurable random variable $X^0$ with values in $E$ we may now formulate the path-dependent SDE in the Banach space $E$, i.e.%
\begin{equation}
\label{eq:SDEinfpath}
dX\left(  t\right)  =\left(  AX\left(  t\right)  +B\left(  t,X\left(
t\right)  \right)  \right)  dt+C\left(  t,X\left(  t\right)  \right)
dW\left(  t\right)\ ,\quad X(0)=X^0\  \text{.}
\end{equation}
The natural concept of solution here would be that of mild solution, but since under our assumptions the stochastic convolution is a priori well defined only in $H$, we consider equation \call{eq:SDEinfpath} in its mild form in the space $H$, that is
\begin{equation}
X\left(  t\right)  =e^{tA}X^{0}+\int_{0}^{t}e^{\left(  t-s\right)  A}B\left(
s,X\left(  s\right)  \right)  ds+\int_{0}^{t}e^{\left(  t-s\right)  A}C\left(
s,X\left(  s\right)  \right)  dW\left(  s\right)\  .\label{SDE in Hilbert}%
\end{equation}
Notice that since $E$ is continuously embedded in $H$ all the integrals in \call{SDE in Hilbert} are meaningful.\\

In the following proposition we clarify the relation between equation \call{SDE in Hilbert} and equation \call{path dep SDE}. 
\begin{proposition}
\label{prop:reformulation}Given an $\mathcal{F}_{0}$-measurable random
vector $y^{0}$  of $\mathds{R}^{d}$, set $X^{0}=\left(  y^{0},y^{0}%
  \ind_{[-T,0)}\right)  ^{T}$. Then, if $\left\{  y\left(  t\right)  \right\}
_{t\in\left[  0,T\right]  }$ is a solution to equation (\ref{path dep SDE}),
the process
\begin{equation}
  \label{eq:Xy}
X\left(  t\right)  =L^ty_t
\end{equation}
is a solution to equation (\ref{SDE in Hilbert}). We also have%
\begin{equation}
y_{t}=M_{t}X(t)\ .\label{eq:yMX}%
\end{equation}
\end{proposition}

\begin{proof}
By \call{eq:etA} the first component of equation (\ref{SDE in Hilbert}) reads%
\begin{align*}
\left(L^ty_t\right)_1=y\left(  t\right)    & =X_{1}\left(  t\right)  =X_{1}^{0}+\int_{0}^{t}B\left(
s,X\left(  s\right)  \right)  _{1}ds+\int_{0}^{t}C\left(  s,X\left(  s\right)
\right)  _{1}dW\left(  s\right)  \\
& =y^{0}+\int_{0}^{t}b\left(  s,M_{s}X\left(  s\right)  \right)
ds+\int_{0}^{t}\sigma\left(  s,M_{s}X\left(  s\right)  \right)  dW\left(
s\right)  \\
& =y^{0}+\int_{0}^{t}b\left(  s,M_{s}L^{s}y_{s}\right)  ds+\int_{0}%
^{t}\sigma\left(  s,M_{s}L^{s}y_{s}\right)  dW\left(  s\right)  \\
& =y^{0}+\int_{0}^{t}b\left(  s,y_{s}\right)  ds+\int_{0}^{t}\sigma\left(
s,y_{s}\right)  dW\left(  s\right),
\end{align*}
which holds true because \textit{it is} equation (\ref{path dep SDE}). About
the second component, we have%
\begin{align*}
\left(  L^{t}y_{t}\right)_2  \left(  r\right)    & =X_{2}\left(  t\right)
\left(  r\right)  =X_{2}^{0}\left(  r+t\right)  1_{\left[  -T,-t\right]
}\left(  r\right)  +X_{1}^{0}1_{\left[  -t,0\right]  }\left(  r\right)  \\
& +\int_{0}^{t}b\left(  s,M_{s}X\left(  s\right)  \right)  1_{\left[
-t+s,0\right]  }\left(  r\right)  ds+\int_{0}^{t}\sigma\left(  s,M_{s}X\left(
s\right)  \right)  1_{\left[  -t+s,0\right]  }\left(  r\right)  dW\left(
s\right)  \\
& =y^{0}+\int_{0}^{t}b\left(  s,y_{s}\right)  1_{\left[  -t+s,0\right]
}\left(  r\right)  ds+\int_{0}^{t}\sigma\left(  s,y_{s}\right)  1_{\left[
-t+s,0\right]  }\left(  r\right)  dW\left(  s\right)  .
\end{align*}
For $r\in\left[  -T,-t\right]  $ this identity reads $y^{0}=y^{0}$, which is
true. For $r\in\left[  -t,0\right]  $ we have
\[
y\left(  t+r\right)  =y_{t}(t+r)=y^{0}+\int_{0}^{t+r}b\left(  s,y_{s}\right)
ds+\int_{0}^{t+r}\sigma\left(  s,y_{s}\right)  dW\left(  s\right),
\]
because $\ind_{\left[  -t+s,0\right]  }\left(  r\right)  =0$ for $s\in\left[
t+r,t\right]  $. This is again a copy of equation (\ref{path dep SDE}). The
proof is complete.
\end{proof}

\begin{remark}
\label{rem:eqnA}
We have seen that, at the level of the mild formulation, the equation in
Hilbert space is just given by two copies of the original SDE. On the
contrary, at the level of the differential formulation, we formally have
\begin{align*}
dX_{1}\left(  t\right)   &  =B\left(  t,M_{t}X\left(  t\right)  \right)
dt+C_{i}\left(  t,M_{t}X\left(  t\right)  \right)  dW\left(  t\right)  \\
dX_{2}\left(  t\right)   &  =\frac{d}{dr}X_{2}\left(  t\right)  dt.
\end{align*}
The first equation, again, is a rewriting of the path-dependent SDE. But the
second equation is just a consistency equation, necessary since we need to
introduce the component $X_{2}\left(  t\right)  $. Here we see one of the
technical problems which motivate this paper: $X_{2}\left(  t\right)
=\left(L^{t}y_{t}\right)_2$ is "never" differentiable (being a trajectory of solution of the
SDE, it has the level of regularity of Brownian motion). In other words,
$X_{2}\left(  t\right)  $ "never" belongs to $D\left(  A\right)  $.
\end{remark}

%From the same arguments as in the proof of proposition \ref{prop:reformulation} it follows that
%Proposition \ref{prop:continuity} then yields that also in this situation the process $X$ takes valued in $\widetilde E$ and has almost surely continuous paths both in $\widetilde E$ and in $H$.

\subsection{It\^o formula for path-dependent functionals}
\label{subsec:Ito_path}
Having introduced the previous infinite-dimensional reformulations, we can apply our abstract result of section \ref{sec:ItoBanach} to obtain a It\^o formula for path-dependent functionals of continuous paths. To this end we recall that we intend to apply theorem \ref{thm:Ito_Banach} to the  spaces
% We apply here the result obtained in the previous section to the framework described in sections \ref{subsec:infinite_reformulation}, \ref{subsec:functionals} in order to deal with path dependent examples.\\
% Consider again the stochastic equation in $\bR^d$
% \begin{equation}
% y\left(  t\right)  =y_{0}+\int_{0}^{t}b\left(  s,y_{s}\right)  ds+\int_{0}^{t}\sigma\left(  s,y_{s}\right)  dW\left(  s\right)  \tag{\ref{path dep SDE}}
% \end{equation}
% where $W$ is a standard brownian motion in $\bR^k$ and now we assume that for each $t\in[0,T]$ $b(t,\cdot)$ maps $C\left([0,t];\bR^d\right)$ into $\bR^d$ and $\sigma(t,\cdot)$ maps $C\left([0,t];\bR^d\right)$ into $\bR^{k\times d}$. 
% With respect to the notations introduced in section \ref{sec:ItoBanach}, choose 
\begin{gather*}
H_1=\bR^d\\
E_2=\left\{\phi\in C\left([-T,0);\bR^d\right)\colon \exists\lim_{s\to 0^-}\phi(s)\in\bR^d\right\}\ \text{,}\\
H_2=L^2\left(-T,0;\bR^d\right)\ \text{,}\\
E=H_1\times E_2\ \text{,}\\
H=H_1\times H_2\ \text{,}\\
\widetilde E=\left\{\xphi{x_1}{x_2}\in E \colon x_1=\lim_{s\to 0^-}x_2(s)\right\}\ \text{,}\\
\widetilde D=\left\{\xphi{x_1}{x_2}\in E \colon x_2\in C^1\left([-T,0);\bR^d\right),\ x_1=\lim_{s\to 0^-}x_2(s)\right\}\ \text{,}\\
U=\bR^k
\end{gather*}
and to the operator $A$ on $H$ given by
\begin{equation*}
  A\xphi{x_1}{x_2}=\xphi{0}{\dot x_2}
\end{equation*}
on the domain
\begin{gather*}
%  D\left(A_E\right)=\left\{\xphi{x_1}{x_2}\in E \colon x_2\in C^1\left([-T,0);\bR^d\right),\ x_1=\lim_{t\to 0^-}x_2(t)\right\}\ \text{,}\\
  D\left(A\right)=\left\{\xphi{x_1}{x_2}\in\bR^d\times W^{1,2}\left(-T,0;\bR^d\right)\colon x_2(0)=x_1\right\}\ .
\end{gather*}
% and set moreover
% \begin{equation*}
%   \widetilde E:=\left\{y=\xphi{x_1}{x_2}\in E \text{ s.t. }x_1=\lim_{s\uparrow 0}x_2(s)\right\}\ \text{.}
% \end{equation*}
% Define the operators $A_H=A_E=A$, $B$ and $C$ as in \call{eq:A}, \call{eq:B}, \call{eq:C}, with
% \begin{gather*}
% D\left(A_H\right)=\left\{\xphi{x_1}{x_2}\in\bR^d\times W^{1,2}\left(-T,0;\bR^d\right)\colon x_2(0)=x_1\right\}\ \text{,}\\
% D\left(A_E\right)=\widetilde E\cap\left(\bR^d\times C^1\left([-T,0);\bR^d\right)\right)\ \text{,}
% \end{gather*}
% so that $A_H:D\left(A_H\right)\to H$, $A_E:D\left(A_E\right)\to E$, $B:[0,T]\times E\to E$ and $C:[0,T]\times E\to L(U;E)$.\\
% Notice that $A$ generates a semigroup $e^{tA}$ which is strongly continuous in $H$ and in $\widetilde E$, but not in $E$; nevertheless it is equibounded in $E$.\\
% We can therefore consider the process with values in $H$ given by 
% \begin{equation}
% X\left(  t\right)  =e^{tA}X^{0}+\int_{0}^{t}e^{\left(  t-s\right)  A}B\left(
% s,X\left(  s\right)  \right)  ds+\int_{0}^{t}e^{\left(  t-s\right)  A}C\left(
% s,X\left(  s\right)  \right)  dW\left(  s\right)\  .\tag{\ref{SDE in Hilbert}}
% \end{equation}
% \todonum{da finire}.......\\
As before $y$ is a continuous process in $\bR^d$ given by
\begin{equation*}
  y(t)= y^0+\int_0^tb(s)\ud s+\int_0^tc(s)\ud W(s),
\end{equation*}
where $W$, $b$ and $c$ are as in subsection \ref{subsec:infinite_reformulation_proc} (we set $Q=Id_{\bR^k}$) and we set
\begin{equation*}
  X(t)=L^ty_t\ .
\end{equation*}
\begin{theorem}
  \label{thm:ItoPath}
Let $\left\{f(t,\cdot)\right\}_{t\in[0,T]}$, $f(t,\cdot):C_t\to\bR$, be a path-dependent functional and define
\begin{gather*}
  F:[0,T]\times E\longrightarrow\bR\\
  F(t,x)=f\left(t,M_tx\right)\ .
\end{gather*}
Suppose that
\begin{enumerate}[label=\emph{(\roman*)}]
\item $F\in C\left([0,T]\times E;\bR\right)$;
\item $F$ is twice differentiable in its second variable with $DF\in C\left([0,T]\times E;E^\ast\right)$ and $D^2F\in C\left([0,T]\times E;L\left(E;E^\ast\right)\right)$;
\item there exists a set $\sT\subset[0,T]$ such that $\lambda\left(\sT\right)=T$ and $F$ is differentiable with respect to $t$ on $\sT\times \widetilde D$;
\item there exists a continuous function $G:[0,T]\times \widetilde E\to\bR$ such that
  \begin{equation*}
    G(t,x)=\frac{\partial F}{\partial t}(t,x)+\langle Ax,DF(t,x)\rangle    
  \end{equation*}
for $(t,x)\in\sT\times \widetilde D$.
\end{enumerate}
Then the identity
\begin{align*}
  f\left(t,y_t\right)&=f\left(0,y_0\right)+\int_0^tG\left(s,X(s)\right)\ud s\\
  &+\int_0^t\left(\langle B(s),DF\left(s,X(s)\right)\rangle+\frac{1}{2}\tr_{\bR^d}\left[C(s)C(s)^\ast D^2F\left(s,X(s)\right)\right]\right)\ud s\\
  &+\int_0^t\langle DF\left(s,X(s)\right),C(s)\ud W(s)\rangle\
\end{align*}
holds in probability.
\end{theorem}
\begin{proof}
  First notice that by proposition \ref{prop:continuity} and the discussion in subsection \ref{subsec:infinite_reformulation_proc} the process $X$ has continuous paths in $\widetilde E$, therefore the set $\left\{X(t)\right\}_{t\in[0,T]}$ is a compact set in $E$.
With this choice of $E$ and $H$, a sequence $\sJ_n:H\to E$ satisfying the requirements of theorem \ref{thm:Ito_Banach} can be constructed (following \cite{FZ15}) in this way:
for any $\epsilon\in\left(0,\nicefrac{T}{2}\right)$ define the function $\tau_\epsilon:[-T,0]\to[-T,0]$ as
\begin{equation*}
  \tau_\epsilon(x)=\begin{cases}-T+\epsilon &\mbox{if } x\in[-T,-T+\epsilon]\\x&\mbox{if } x\in[-T+\epsilon,-\epsilon]\\-\epsilon&\mbox{if }x\in[-\epsilon,0].\end{cases}
\end{equation*}
Then choose any $C^{\infty}(\bR;\bR)$ function $\rho$ such that $\VV\rho\VV_1=1$, $0\leq\rho\leq 1$ and $\supp(\rho)\subseteq[-1,1]$ and define a sequence $\left\{\rho_n\right\}$ of mollifiers by $\rho_n(x):=n\rho(nx)$. Set, for any $\phi\in L^2(-T,0;\bR^d)$
\begin{equation}
\label{eq:Jn}
  \rJ_n\phi(x):=\int_{-T}^0\rho_n\big(\rho_{2n}\ast\tau_{\frac{1}{n}}(x)-y\big)\phi(y)\ud y
\end{equation}
and finally define $\sJ_n$ as
\begin{equation*}
  \sJ_n\xphi{x}{\phi}:=\left(\begin{matrix}x\\ \rJ_n\phi\end{matrix}\right)\ \text{.}
\end{equation*}
The proof is then completed applying theorem \ref{thm:Ito_Banach} to the function $F$ and its extension $G$.
\end{proof}
\begin{remark}
\label{rem:need}
  The choice of the spaces $H$ and $E$ and of the operators $A$ and $\sJ_n$ does not depend on $F$ and is the same for all path-dependent functionals of continuous processes. The only assumptions that need to be checked on each functional are the regularity conditions and existence of the extension $G$.
\end{remark}
The path-dependent functional given in \call{example 2} is not covered by the previous result since it is not jointly continuous on $[0,T]\times E$. However it satisfies the assumptions of corollary \ref{coro:piecewise}, which we now state in its path-dependent formulation.
\begin{corollary}
  Let $f$ and $F$ be as in theorem \ref{thm:ItoPath}. If $F$ satisfies the assumptions of corollary \ref{coro:piecewise}, then the formula
\begin{align*}
  f\left(T,y_T\right)&=F\left(0,y_0\right)+\sum_{j=1}^n\left[f\left(t_j,y_{t_j}\right)-f\left(t_j-,y_{t_j}\right)\right]\\
  &+\int_0^TG\left(s,X(s)\right)\ud s+\int_0^T\langle DF\left(s,X(s)\right),C(s)\ud W(s)\rangle\\
  &+\int_0^T\left(\langle B(s),DF\left(s,X(s)\right)\rangle+\frac{1}{2}\tr_{H_1}\left[C(s)QC(s)^\ast D^2F\left(s,X(s)\right)\right]\right)\ud s
\end{align*}
holds in probability.
\end{corollary}

% The following result follows now immediately from theorem \ref{thm:Ito_}:
% \begin{corollary}
%   Let $f(t,\cdot)$, $t\in[0,T]$, $f(t,\cdot):C\left([0,t];\bR^d\right)\to\bR$, be a non-anticipative functional such that the associated map $F$, as defined by \call{eq:F}, satisfies the assumption of theorem \ref{thm:Ito_Banach}. If there exists an extension $G$ as in \call{eq:G_banach}, then\todonum{da sistemare con notazione $\bR^d$}
% in probability.
% \end{corollary}
% Furthermore, if the assumptions of theorem \ref{thm:comparison} are satisfied, we recover the Ito formula proved in \cite{ConFou}:
% \begin{theorem}
%   ...
% \end{theorem}
\section{Application to Kolmogorov equations}
\label{sec:Kolmogorov}
\subsection{Uniqueness of solutions}
\label{subsec:uniqueness}
% Consider the backward Kolmogorov path-dependent equation
% \begin{equation}
% \label{eq:PHVKolmogorov}
% \begin{cases}\rD_t\nu(\gamma_t)+b_t(\gamma_t)\cdot\rD\nu_t(\gamma_t)+\frac{1}{2}\tr\left[\sigma_t\left(\gamma_t\right)\sigma_t\left(\gamma_t\right)^\ast\rD_j^2\nu_t(\gamma_t)\right]=0\ \text{,}\\
% \nu_T(\gamma_T)=f(\gamma_T)\ \text{.}\end{cases}
% \end{equation}
% \begin{theorem}
%   Assume that a solution to equation \call{eq:PHVKolmogorov} exists. Then it is unique.
% \end{theorem}
% \begin{proof}
% Lift\todonum{da scrivere bene} the equation and the final condition to the infinite-dimensional setting. Then the equation itself provides the extension therefore by It\^o formula there is uniqueness for the infinite-dimensional equation. Finally theorem \ref{comparison} allows to obtain uniqueness for the original equation.
% \end{proof}
% \begin{remark}
%   Existence of solutions to the lifted Kolmogorov equation and, consequently, to the original equation has been investigated in \cite{FZ15} when the diffusion coefficient $\sigma$ is constant.
% The proof of the result above provides uniqueness for the lifted equations as well\todonum[fancyline]{forse si può mettere più in luce questo punto, magari organizzando questa sezione come un unico discorso invece che come un elenco di teoremi}.
We begin investigating Kolmogorov equation in the abstract setting of section \ref{sec:ItoBanach}, discussing the particular case of path-dependent Kolmogorov equations afterward.
Let the spaces $H$, $E$, $\widetilde E$, $\widetilde D$ and $U$, the Wiener process $W$ and the operator $A$ be as in section \ref{sec:ItoBanach}; given $B:[0,T]\times \widetilde E\to E$ and $C:[0,T]\times \widetilde E\to L\left(U;H_1\times\left\{0\right\}\right)$ we can consider the partial differential equation
 \begin{equation}
 \label{eq:PKolmogorov}
\begin{cases}
\frac{\partial V}{\partial t}(t,x)+\left\langle DV(t,x),Ax+B\left(  t,x\right)\right\rangle+\frac{1}{2}\tr_{H_1}\left(  C(s,x) Q C(s,x)^\ast D^{2}V(t,x)\right) =0\ ,\\ V(T,\cdot)=\Phi\ \text{,}
\end{cases}
\end{equation}
where the terminal condition $\Phi$ is chosen in $C_b^{2,\alpha}(E)$, the space of twice differentiable real-valued functions $\phi$ on $E$ such that $\phi$, $D\phi$ and $D^2\phi$ are bounded and the map $E\ni x\to D^2\phi(x)\in L\left(E;E^\ast\right)$ is $\alpha$-H\"older continuous. \\
\begin{definition}
\label{def:solE}
We say that a function $V:[0,T]\times E\to\bR$ is a \emph{classical solution} to equation \call{eq:PKolmogorov} in $E$ if 
\begin{equation*}
V\in L^{\infty}\left(  0,T;C_{b}^{2,\alpha}\left(E;\mathds{R}%
\right)  \right)  \cap C\left(  \left[  0,T\right]  \times E;\mathds{R}%
\right)\ \text{,}
\end{equation*}
$V$ is differentiable with respect to $t$ on $\sT\times \widetilde D$, $\sT\subset[0,T]$ being a set of full measure, and satisfies identity \call{eq:PKolmogorov} for every $t\in\sT  $ and $x\in \widetilde D$.
\end{definition}
Assume that $B$ and $C$ are continuous and such that the stochastic SDE
\begin{equation}
  \label{eq:SDE6}
  \ud X(s)=AX(s)+B\left(s,X(s)\right)\ud s+ C\left(s,X(s)\right)\ud W(s)\ \text{ for }s\in[t,T],\quad X(t)=x
\end{equation}
has a mild solution $X^{t,x}$ in $H$ for all $t\in[0,T]$ and all $x\in \widetilde E$, such that $X^{t,x}(s)$ belongs to $\widetilde E$ for all $s\in[t,T]$ and that the set $\left\{X^{t,x}(s)\right\}_{s\in[t,T]}$ is almost surely relatively compact in $E$.\\
%(for example one can assume that the local bounds of $B$ and $C$ in the variable $x$ are uniform in $t$ and that that the processes $B\left(s,X(s)\right)$ and $C\left(s,X(s)\right)$ are progressively measurable, compare section \ref{subsec:infinite_reformulation_SDE})

\begin{theorem}
\label{thm:uniqueness}
Under the above assumptions any classical solution to equation \call{eq:PKolmogorov} is uniquely determined on the space $\widetilde E$
\end{theorem}
\begin{proof}
Suppose there exists a solution $V$. Since $DV$, $D^2V$, $B$ and $C$ are defined on $[0,T]\times \widetilde E$ and are continuous, the function
\begin{equation*}
  G(t,x)=-\left\langle B(t,x),DV(t,x)\right\rangle-\frac{1}{2}\tr_{\bR^d}\left[C(t,x)C(t,x)^\ast D^2V(t,x)\right]
\end{equation*}
is a continuous extension of
\begin{equation*}
  \frac{\partial V}{\partial t}(t,x)+\left\langle Ax,DV(t,x)\right\rangle
\end{equation*}
from $\sT\times \widetilde D$ to $[0,T]\times \widetilde E$, because $V$ satisfies Kolmogorov equation.\\
Therefore we can apply theorem \ref{thm:Ito_Banach} to obtain
\begin{align*}
  \Phi\left(X^{t,x}(T)\right)&=V\left(t,X^{t,x}(t)\right)+\int_t^TG\left(s,X^{t,x}(s)\right)\ud s+\int_t^T\langle B\left(s,X^{t,x}(s)\right),DV\left(s,X^{t,x}(s)\right)\rangle\ud s\\
  &\phantom{=}+\frac{1}{2}\int_t^T\tr_{H_1}\left[C\left(s,X^{t,x}(s)\right)C\left(s,X^{t,x}(s)\right)^\ast D^2V\left(s,X^{t,x}(s)\right)\right]\ud s\\
  &\phantom{=}+\int_t^T\langle DV\left(s,X^{t,x}(s)\right),C\left(s,X^{t,x}(s)\right)\ud W(s)\rangle\\
  &=V\left(t,X^{t,x}(t)\right)+\int_t^T\langle DV\left(s,X^{t,x}(s)\right),C\left(s,X^{t,x}(s)\right)\ud W(s)\rangle\ .
\end{align*}
The integral in the last line is actually a stochastic integral in a Hilbert space, since for every $u\in U$ $C(s,x)u$ belongs to $H_1\times\{0\}$; taking expectations in the previous identity we obtain that
\begin{equation*}
  V(t,x)=\bE\left[\Phi\left(X^{t,x}(T)\right)\right]\ .
\end{equation*}
\end{proof}

In the path-dependent case the above discussion can be rephrased as follows. Choose the spaces $H$, $E$, $\widetilde E$, $\widetilde D$, $U$ and the operator $A$ as in section \ref{sec:infinite_reformulation}. Let moreover $D\left([a,b];\bR^d\right)$ denote the space of $\bR^d$-valued càdlàg functions on the interval $[a,b]$, equipped with the supremum norm and set
\begin{equation*}
  O=\bR^d\times\left\{\phi\in D\left([-T,0);\bR^d\right)\colon \exists \lim_{s\to 0^-}\phi(s)\in\bR^d\right\}\ .
\end{equation*}
Then $E\subset O\subset H$ and $O$ is isomorphic to $D\left([-T,0];\bR^d\right)$. Through \call{eq:defM} we can define $M_T$ as a map from $O$ to $D\left([0,T];\bR^d\right)$.\\
For a given continuous path $\gamma_t\in C\left([0,t];\bR^d\right)$ consider the stochastic differential equation in $\bR^d$
\begin{equation}
  \label{eq:SDEt}
  \ud y(s)=b\left(s,y_s\right)\ud s+\sigma\left(s,y_s\right)\ud W(s)\ ,\quad s\in[t,T] ,\quad y_t=\gamma_t,
\end{equation}
and assume that $b$ and $\sigma$ are regular enough for equation \call{eq:SDEt} to have a continuous solution $y^{\gamma_t}$ (compare subsection \ref{subsec:infinite_reformulation_SDE}).\\
Choose $f\in C_b^{2,\alpha}\left(D\left([0,T];\bR^d\right)\right)$ and define 
\begin{gather}
\nonumber \Phi:O\to \bR \\ 
\label{eq:Phif}\Phi\left(x\right)=f\left(M_Tx\right)\ .
\end{gather}
If the operators $B$, $C$ are defined from $b$ and $\sigma$ as in \call{eq:B}, \call{eq:C}, we can consider the infinite-dimensional Kolmogorov backward equation
\begin{equation}
\label{eq:KolmPath}
\begin{cases}
\frac{\partial V}{\partial t}(t,x)+\left\langle DV(t,x),Ax+B\left(  t,x\right)\right\rangle+\frac{1}{2}\tr_{\bR^d}\left(  C(s,x)C(s,x)^\ast D^{2}V(t,x)\right) =0\ ,\\ V(T,\cdot)=\Phi\ \text{.}
\end{cases}
\end{equation}
We call equation \call{eq:KolmPath} the \emph{path-dependent Kolmogorov backward equation associated to the triple $(b,\sigma,f)$}.
A classical solution $V$ to equation \call{eq:KolmPath} uniquely identifies a path-dependent functional $v$, which is given by
\begin{equation}
\label{eq:defv}
  v\left(t,\gamma_t\right)=V\left(t,L^t\gamma_t\right)\ .
\end{equation}
Since $L^t\gamma_t\in \widetilde E$ if and only if $\gamma_t\in C\left([0,t];\bR^d\right)$, it is an immediate consequence of Theorem \ref{thm:uniqueness} that for every $t$ the function $v(t,\cdot)$ given by \call{eq:defv} is uniquely determined on $C\left([0,t];\bR^d\right)$.\\
Therefore our uniqueness result for path-dependent Kolmogorov equations takes the following form.
\begin{theorem}
  Let $b$ and $\sigma$ such that equation \call{eq:SDEt} has a continuous solution for every $t\in[0,T]$ and every $\gamma_t\in C\left([0,t];\bR^d\right)$. % and such that the associated operators $B$ and $C$ are regular enough for equation \call{eq:SDE6} to have a unique mild solution in $H$ for every $t\in[0,T]$ and every $x\in \widetilde E$.
Then, for any $f\in C_b^{2,\alpha}\left(D\left([0,T];\bR^d\right)\right)$, any path-dependent functional $v$ such that the function $V(t,x)=v\left(t,M_tx\right)$ is a classical solution to the path-dependent Kolmogorov backward equation associated to $(b,\sigma,f)$ is uniquely determined on $C\left([0,t];\bR^d\right)$.
\end{theorem}
\begin{proof}
  Set $\tilde v\left(t,\gamma_t\right)=V\left(t,L^t\gamma_t\right)$. By the definition of $M_T$, the restriction of $\Phi$ to $E$ belongs to $C_b^{2,\alpha}\left(E\right)$. Thanks to the assumptions on $b$ and $\sigma$, for any $\gamma_t\in C\left([0,t];\bR^d\right)$ there exists a mild solution $X^{t,x}$ to equation \call{eq:SDE6} with $x=L^t\gamma_t\in\widetilde E$. By propositions \ref{prop:reformulation} and \ref{prop:continuity} $X^{t,x}$ takes values in $\widetilde E$ and has continuous paths with respect to the topology of $E$.\\
Therefore we can apply theorem \ref{thm:uniqueness}, hence $\tilde v$ is uniquely determined on $C\left([0,t];\bR^d\right)$; but for any $\gamma_t\in C_t$
\begin{equation*}
  v\left(t,\gamma_t\right)=v\left(t,M_tL^t\gamma_t\right)=V\left(t,L^t\gamma_t\right)=\tilde v\left(t,\gamma_t\right)\ .
\end{equation*}
\end{proof}
% \begin{proof}
% Since $M_T$ is simply a traslation, $\Phi_E$ belongs to $C_b^{2,\alpha}\left(E\right)$. Thanks to the assumptions on $b$ and $\sigma$, for any $\gamma_t\in C\left([0,t];\bR^d\right)$ there exists a mild solution $X^{t,x}$ to equation \call{eq:SDE6} with $x=L^t\gamma_t\in\widetilde E$. By propositions \ref{prop:reformulation} and \ref{prop:continuity} $X^{t,x}$ takes values in $\widetilde E$ and has continuous paths with respect to the topology of $E$.\\
% Let $v^1$ and $v^2$ be two classical solution to the path-dependent Kolmogorov equation both associated to $(b,\sigma,f)$ and define the functions $V^1(t,x)=v^1\left(t,M_tx\right)$ and $V^2(t,x)=v^2\left(t,M_tx\right)$. Then by definition \ref{def:solP} and remark \ref{rem:restriction} $V^1_E$ and $V^2_E$ both solve equation \call{eq:KolmPath} hence by theorem \ref{thm:uniqueness} they coincide on $\widetilde E$. Since for any $\gamma_t\in C\left([0,t];\bR^d\right)$ we have that, for $j=1,2$,
% \begin{equation*}
%   v^j\left(t,\gamma_t\right)=v^j\left(t,M_tL^t\gamma_t\right)=V^j_E\left(t,L^t\gamma_t\right)
% \end{equation*}
% and $L^t\gamma_t\in\widetilde E$, $v^1(t,\cdot)$ and $v^2(t,\cdot)$ coincide on $C\left([0,t];\bR^d\right)$.
% \end{proof}
\begin{remark}
  For $\gamma_t\in C\left([0,t];\bR^d\right)$ and $x=\left(\gamma_t(t),L^t\gamma_t\right)^T$ the process $X^{t,x}$ in the previous proof is given by
\begin{equation*}
  X^{t,x}(s)=L^sy^{\gamma_t}_s\ .
\end{equation*}
Therefore if $v$ and $V$ are as above we have, by the definition of solution and theorem \ref{thm:uniqueness}, that
\begin{align*}
  v\left(t,\gamma_t\right)&=V\left(t,L^t\gamma_t\right)\\
  &=\bE\left[\Phi\left(X^{t,x}(T)\right)\right]\\
  &=\bE\left[f\left(M_TX^{t,x}(T)\right)\right]\\
  &=\bE\left[f\left(M_TL^Ty^{\gamma_t}_T\right)\right]\\
  &=\bE\left[f\left(y^{\gamma_t}_T\right)\right]\ .
\end{align*}
This is what one would expect to be the solution to a Kolmogorov equation with terminal condition $f$ associated (in some sense) to the SDE \call{eq:SDEt}.
\end{remark}
Notice that the extension of $\gamma_t$ introduced by the operator $L^t$ is arbitrary; nevertheless it does not play any role in the path-dependent Kolmogorov equation since $B$ and $C$ are defined using $M_t$, compare remark \ref{rem:ML}.
\subsection{Another example on Kolmogorov Equations}
\label{sec:example}
We try to identify a class of functions solving virtually a Kolmogorov type
equations. The inspiration comes from \cite{DGR}, section 9.9,
see also \cite{Cos-Rus-2100}, Theorem 3.5 for a variant.

Let $N\in \mathds{N}$, $g_{1},...,g_{N}\in BV\left( \left[ 0,T\right]
\right) $. We set $g_{0}=1$. We define by $\Sigma \left( t\right) $ the $%
\left( N+1\right) \times \left( N+1\right) $ matrix%
\[
\Sigma _{ij}\left( t\right) :=\int_{t}^{T}g_{i}\left( s\right) g_{j}\left(
s\right) ds.
\]%
We suppose that $\Sigma \left( t\right) $ \ is invertible for any $0\leq t<T$%
. We denote by%
\[
p_{t}\left( x\right) =\frac{1}{\left( 2\pi \right) ^{\frac{N+1}{2}}\sqrt{%
\det \Sigma \left( t\right) }}\exp \left( -\frac{1}{2}x^{T}\Sigma
^{-1}\left( t\right) x\right) 
\]%
the Gaussian density with covariance $\Sigma \left( t\right) $, for $t\in
\lbrack 0,T)$, $x\in \mathds{R}^{N+1}$. Let $f:\mathds{R}^{N+1}\rightarrow 
\mathds{R}$ be a continuous function with polynomial growth. We set%
\[
\widehat{g}_{j}\left( s\right) =g_{j}\left( s+T\right) 
\]%
$0\leq j\leq N$, $s\in \left[ -T,0\right] $. We consider $H:C\left( \left[
-T,0\right] \right) \rightarrow \mathds{R}$ defined by%
\[
H\left( \eta \right) =f\left( \eta \left( 0\right) ,\int_{\left[ -T,0\right]
}\widehat{g}_{1}d\eta ,...,\int_{\left[ -T,0\right] }\widehat{g}_{N}d\eta
\right) 
\],
where 
\[
\int_{\left[ -T,0\right] }\widehat{g}_{i}d\eta :=\widehat{g}_{i}\left(
0\right) \eta \left( 0\right) -\int_{\left[ -T,0\right] }\eta d\widehat{g}%
_{i}.
\]%
To simplify, let us assume $g_{i}$ continuous.

We define $\mathcal{U}:\left[ 0,T\right] \times \mathds{R}\times C\left( %
\left[ -T,0\right] \right) \rightarrow \mathds{R}$ by%
\begin{equation}
\mathcal{U}\left( t,x,\psi \right) =\widetilde{\mathcal{U}}\left( t,x,\int_{%
\left[ -T,0\right] }g_{1}\left( \cdot +t\right) d\psi ,...,\int_{\left[ -T,0%
\right] }g_{N}\left( \cdot +t\right) d\psi \right),   \label{def U}
\end{equation}%
where $\widetilde{\mathcal{U}}:\left[ 0,T\right] \times \mathds{R}\times 
\mathds{R}^{N}\rightarrow \mathds{R}$ is motivated by the following lines.

We consider the martingale%
\[
M_{t}=E\left[ h|\mathcal{F}_{t}\right] 
\]%
where (with $\widehat{W}_{s}=W_{s+T}$, $s\in \left[ -T,0\right] $), 
\[
h=H\left( \widehat{W}\right) =f\left( W_{T},\int_{0}^{T}g_{1}\left( s\right)
dW_{s},...,\int_{0}^{T}g_{N}\left( s\right) dW_{s}\right) .
\]%
We proceed by a finite-dimensional analysis. 

We remind that $\widehat{g}_{j}\left( s\right) =g_{j}\left( s+T\right) $. We
evaluate more specifically the martingale $M$. We get%
\[
M_{t}=\widetilde{\mathcal{U}}\left( t,W_{t},\int_{0}^{t}g_{1}\left( s\right)
dW_{s},...,\int_{0}^{t}g_{N}\left( s\right) dW_{s}\right),
\]%
where%
\begin{eqnarray*}
\widetilde{\mathcal{U}}\left( t,x,x_{1},...,x_{N}\right)  &=&E\left[ f\left(
x+W_{T}-W_{t},x_{1}+\int_{t}^{T}g_{1}dW,...,x_{N}+\int_{t}^{T}g_{N}dW\right) %
\right]  \\
&=&\int_{\mathds{R}^{N+1}}f\left( x+\xi _{0},x_{1}+\xi _{1},...,x_{N}+\xi
_{N}\right) p_{t}\left( \xi \right) d\xi  \\
&=&\int_{\mathds{R}^{N+1}}f\left( \xi _{0},\xi _{1},...,\xi _{N}\right)
p_{t}\left( x-\xi _{0},x_{1}-\xi _{1},...,x_{N}-\xi _{N}\right) d\xi
_{0}d\xi _{1}...d\xi _{N}.
\end{eqnarray*}%
By inspection we can show, see also \cite{DGR}, that $\widetilde{\mathcal{%
U}}\in C^{1,2}\left( [0,T)\times \mathds{R}^{N+1}\right) $%
\begin{eqnarray}
\partial _{t}\widetilde{\mathcal{U}}+\frac{1}{2}\sum_{i,j=0}^{N}\Sigma
_{ij}\left( t\right) \frac{\partial ^{2}\widetilde{\mathcal{U}}}{\partial
x_{i}\partial x_{j}} &=&0  \label{Kolmog} \\
\widetilde{\mathcal{U}}\left( T,x\right)  &=&f\left( x\right),   \nonumber
\end{eqnarray}%
where $x=\left( x_{0},x_{1},...,x_{N}\right) $. This can be done via the
property of the density kernel $\left( t,\xi \right) \mapsto p_{t}\left( \xi
\right) $ and classical integration theorems. We set $\mathcal{U}:\left[ 0,T%
\right] \times \mathds{R}\times C\left( \left[ -T,0\right] \right)
\rightarrow \mathds{R}$ as in (\ref{def U}).

\begin{proposition}
Let $C^{2}:=C^{2}\left( \left[ -T,0\right] \right) $. The map $\mathcal{U}$
has the following properties:

i) $\mathcal{U\in }C^{0,2,0}$

ii) $\mathcal{U\in }C^{1,2,1}\left( \left[ 0,T\right] \times \mathds{R}%
\times C^{2}\right) $

iii) the map%
\[
\left( t,x,\psi \right) \longmapsto \mathcal{A}\left( \mathcal{U}\right)
\left( t,x,\psi \right) :=\partial _{t}\mathcal{U}\left( t,x,\psi \right)
+\left\langle D^{\psi }\mathcal{U}\left( t,x,\psi \right) ,\psi ^{\prime
}\right\rangle 
\]%
extends continuously on $\left[ 0,T\right] \times \mathds{R}\times C\left( %
\left[ -T,0\right] \right) $ to an operator still denoted by $\mathcal{A}%
\left( \mathcal{U}\right) \left( t,x,\psi \right) $

iv) 
\[
\partial _{t}\mathcal{U+A}\left( \mathcal{U}\right) +\frac{1}{2}\partial
_{xx}^{2}\mathcal{U}=0.
\]
\end{proposition}

\begin{proof}
i) Obvious.

 ii) We evaluate the different derivatives for $\left( t,x,\psi \right) \in %
\left[ 0,T\right] \times \mathds{R}\times C^{2}$. We get from (\ref{def U})%
\[
\partial _{t}\mathcal{U}\left( t,x,\psi \right) =\partial _{t}\widetilde{%
\mathcal{U}}\left( t,x,\int_{\left[ -t,0\right] }g_{1}\left( \cdot +t\right)
d\psi ,...,\int_{\left[ -t,0\right] }g_{N}\left( \cdot +t\right) d\psi
\right) 
\]%
\begin{equation}
+\sum_{j=1}^{N}\partial _{j}\widetilde{\mathcal{U}}\left( t,x,\int_{\left[
-t,0\right] }g_{1}\left( \cdot +t\right) d\psi ,...,\int_{\left[ -t,0\right]
}g_{N}\left( \cdot +t\right) d\psi \right) \frac{d}{dt}\int_{\left[ -t,0%
\right] }g_{j}\left( \cdot +t\right) d\psi .  \label{ident 3}
\end{equation}%
Now we observe that 
\begin{eqnarray*}
\frac{d}{dt}\int_{\left[ -t,0\right] }g_{j}\left( \cdot +t\right) d\psi  &=&%
\frac{d}{dt}\int_{\left[ -t,0\right] }g_{j}\left( \xi \right) \psi ^{\prime
}\left( \xi -t\right) d\xi  \\
&=&g_{j}\left( t\right) \psi ^{\prime }\left( 0\right) -\int_{\left[ -t,0%
\right] }g_{j}\left( \xi \right) \psi ^{\prime \prime }\left( \xi -t\right)
d\xi 
\end{eqnarray*}%
\begin{equation}
=\int_{\left[ -t,0\right] }\psi ^{\prime }\left( \xi -t\right) g_{j}^{\prime
}\left( d\xi \right)   \label{ident 4}
\end{equation}%
(remark that, without restriction of generality, we can take $g_{j}\left(
0\right) =0$). Now we calculate%
\begin{eqnarray}
&&\left\langle D^{\psi }\mathcal{U}\left( t,x,\psi \right) ,\psi ^{\prime
}\right\rangle   \label{ident 5} \\
&=&\sum_{j=1}^{N}\partial _{j}\widetilde{\mathcal{U}}\left( t,x,\int_{\left[
-t,0\right] }g_{1}\left( \cdot +t\right) d\psi ,...,\int_{\left[ -t,0\right]
}g_{N}\left( \cdot +t\right) d\psi \right) \left\langle D^{\psi }\int_{\left[
-t,0\right] }g_{j}\left( \cdot +t\right) d\psi ,\psi ^{\prime }\right\rangle
.  \nonumber
\end{eqnarray}%
Now the application%
\begin{eqnarray*}
\psi  &\longmapsto &\int_{\left[ -t,0\right] }g_{j}\left( \cdot +t\right)
d\psi =\int_{\left[ -t,0\right] }g_{j}\left( \xi +t\right) \psi ^{\prime
}\left( \xi \right) d\xi  \\
&=&-\int_{\left[ -t,0\right] }d\psi \left( \xi \right) \int_{(\xi
+t,0]}dg_{j}\left( l\right) =-\int_{(0,t]}dg_{j}\left( l\right)
\int_{[l-t,0)}d\psi \left( \xi \right)  \\
&=&-\int_{(0,t]}dg_{j}\left( l\right) \psi \left( l-t\right) 
\end{eqnarray*}%
has to be differentiated in the direction $\psi ^{\prime }$. Taking into
account (\ref{ident 3}), (\ref{ident 4}), (\ref{ident 5}), it follows that 
\begin{equation}
\partial _{t}\mathcal{U}\left( t,x,\psi \right) +\left\langle D^{\psi }%
\mathcal{U}\left( t,x,\psi \right) ,\psi ^{\prime }\right\rangle =\partial
_{t}\widetilde{\mathcal{U}}\left( t,x,\int_{\left[ -t,0\right] }g_{1}\left(
\cdot +t\right) d\psi ,...,\int_{\left[ -t,0\right] }g_{N}\left( \cdot
+t\right) d\psi \right)   \label{ident 6}
\end{equation}%
for every $\psi \in C^{2}$. On the other hand by (\ref{Kolmog}) it follows
that $\mathcal{U\in }C^{1,2,1}\left( \left[ 0,T\right] \times \mathds{R}%
\times C^{2}\right) $.

 iii) By (\ref{ident 6}), for $\left( t,x,\psi \right) \in \left[ 0,T\right]
\times \mathds{R}\times C^{2}$, we get%
\[
\mathcal{A}\left( \mathcal{U}\right) =\partial _{t}\widetilde{\mathcal{U}}%
\left( t,x,\int_{\left[ -t,0\right] }g_{1}\left( \cdot +t\right) d\psi
,...,\int_{\left[ -t,0\right] }g_{N}\left( \cdot +t\right) d\psi \right) .
\]

 iv) This claim follows by inspection, taking into account (\ref{Kolmog}).
\end{proof}

\begin{acknowledgements}
The second named author benefited partially from the support of the ``FMJH Program Gaspard Monge in optimization and operation research'' (Project 2014-1607H). He is also grateful for the invitation to the Department of Mathematics of the University of Pisa.
The third named author is grateful for the invitation to ENSTA. 
\end{acknowledgements}

% BibTeX users please use one of
%\bibliographystyle{spbasic}      % basic style, author-year citations
%\bibliographystyle{spmpsci}      % mathematics and physical sciences
%\bibliographystyle{spphys}       % APS-like style for physics
%\bibliography{}   % name your BibTeX data base

\bibliographystyle{spmpsci}
\bibliography{biblio}

% Non-BibTeX users please use
%\begin{thebibliography}{}
%
% and use \bibitem to create references. Consult the Instructions
% for authors for reference list style.
%
%\bibitem{RefJ}
% Format for Journal Reference
%Author, Article title, Journal, Volume, page numbers (year)
% Format for books
%\bibitem{RefB}
%Author, Book title, page numbers. Publisher, place (year)
% etc
%\end{thebibliography}

\end{document}